\documentclass[12pt]{article}
\usepackage{latexsym, amsmath, amsfonts, amsthm, amssymb}
\usepackage{times}
\usepackage{a4wide}
\usepackage{times}
\usepackage{graphicx, tocloft}
\usepackage{mathrsfs}
\usepackage{url, hyperref}

\def \tr {\rm{Tr}}

\def \Vol {{ \rm Vol }}

\def \qqquad {\qquad \qquad }
\def \RR {\mathbb R}

\def \EE {\mathbb E}

\def \CC {\mathbb C}
\def \PP {\mathbb P}

\def \eps {\varepsilon}

\def \cF {\mathcal F}

\def \Tr {{ \rm Tr }}

\newtheorem{theorem}{Theorem}[section]

\newtheorem{lemma}[theorem]{Lemma}

\newtheorem{proposition}[theorem]{Proposition}
\newtheorem{corollary}[theorem]{Corollary}

\theoremstyle{definition}
\newtheorem{remark}[theorem]{Remark}

\def\qed{\hfill $\vcenter{\hrule height .3mm
		\hbox {\vrule width .3mm height 2.1mm \kern 2mm \vrule width .3mm
			height 2.1mm} \hrule height .3mm}$ \bigskip}

\def \id {{\rm Id}}

\def \cov {{ \rm Cov}}
\def \var {{ \rm Var}\,}
\def \supp {{ \rm supp}}

\begin{document}

\title{Distances between non-symmetric convex bodies: optimal bounds up to polylog}

\author{Pierre Bizeul
and Boaz Klartag}

\date{}
\maketitle

\begin{abstract} In this paper we determine, up to polylogarithmic factors, the diameter of the
Banach--Mazur compactum of $n$-dimensional convex bodies without symmetry
assumptions. We prove that for any convex bodies $K_1,K_2\subset \mathbb{R}^n$,
\begin{equation}
d_{BM}(K_1,K_2)\le Cn\log^\alpha(n+1),
\label{eq_1624}
\end{equation}
for universal constants $C,\alpha>0$, improving an earlier bound of Rudelson.
We also study the partial-containment distance $d_{PC}$, in which the Banach-Mazur requirement to contain the other body in its entirety is relaxed to $99\%$-containment.
We prove that this relaxation leads to a very different behavior:
\begin{equation}
d_{PC}(K_1,K_2) \leq C \log^{\alpha} (n+1)
\label{eq_1625}
\end{equation}
for all convex bodies $K_1,K_2 \subseteq \mathbb{R}^n$. This demonstrates that in high dimensions, any convex body is not too far from an affine image of any other convex body, when we look at the bulk of their mass.

\medskip
In the centrally-symmetric case, the optimal upper bound for the Banach-Mazur distance is obtained in the John position. In contrast, our proofs rely on the isotropic position. The analytic core of our argument
is a two-sided comparison, in the gauge order, between isotropic log-concave
measures and Gaussian measures. This yields a new isotropic $M$-bound that
complements E. Milman's $M^*$-bound. We also provide applications to linear
symplectic geometry and to the first Dirichlet eigenvalue of the Laplacian.
\end{abstract}

\section{Introduction}
\label{intro}

\subsection{Main results}

Let $n \geq 2$.
The  Banach-Mazur distance
between two convex bodies $K_1, K_2 \subseteq \RR^n$, denoted by $d_{BM}(K_1, K_2)$, is
the infimum over all $\lambda \geq 1$ for which there exists an invertible
linear map $T: \RR^n \rightarrow \RR^n$ and vectors $x,y \in \RR^n$ with
\begin{equation}
 y + K_2 \subseteq T(x + K_1) \subseteq \lambda (y + K_2). \label{eq_1018}
\end{equation}
One of the basic problems in asymptotic convex geometry is to estimate the
largest possible value of $d_{BM}(K_1,K_2)$, as $K_1$ and $K_2$ range over all
convex bodies in $\RR^n$. It is well-known that for any convex body $K \subseteq \RR^n$ (i.e., a compact, convex set with a non-empty interior),
\begin{equation}  d_{BM}(K, B^n) \leq n \label{eq_308} \end{equation}
with equality if and only if $K$ is a simplex, where $B^n = \{ x \in \RR^n \, ; \, |x| \leq 1 \}$ is the unit Euclidean ball.
The proof of (\ref{eq_308}) goes back to John \cite{J} and the equality case is established in Palmon \cite{palmon}.
The bound (\ref{eq_308})  implies that $d_{BM}(K_1, K_2) \leq n^2$ for any pair of convex bodies $K_1, K_2 \subseteq \RR^n$.
At the end of the 1990s, this was improved by Rudelson \cite{rudelson} to the bound
\begin{equation}
 d_{BM}(K_1, K_2) \leq C n^{4/3} \log^{\alpha} n, \label{eq_1205} \end{equation}
where $C, \alpha > 0$ are universal constants. The methods employed in Rudelson \cite{rudelson} and in Banaszczyk, Litvak, Pajor and Szarek \cite{BLPS} at the time involved mostly tools from the local theory of Banach spaces such as the $M M^*$ product. In contrast, our point of view also emphasizes the high level of regularity enjoyed by the uniform measure on a high-dimensional convex body. We prove the following estimate, which determines the diameter of the non-symmetric Banach-Mazur
compactum up to polylogarithmic factors:

\begin{theorem} For any convex bodies $K_1, K_2 \subseteq \RR^n$,
\begin{equation}  d_{BM}(K_1, K_2) \leq C n \cdot \log^{\alpha} n, \label{eq_1042} \end{equation}
where $C, \alpha > 0$ are universal constants.
\label{thm0}
\end{theorem}

Our proof of Theorem~\ref{thm0} yields $\alpha \leq 4$.
The exponent $\alpha$ can be slightly improved; see Remark~\ref{rmk_exponent}.
The bound of Theorem \ref{thm0} is optimal up to polylogarithmic factors.
In the case where $K_1, K_2 \subseteq \RR^n$ are centrally-symmetric, i.e., $K_i = -K_i$,
John's theorem implies that $d_{BM}(K_i, B^n) \leq \sqrt{n}$ and hence $d_{BM}(K_1, K_2) \leq n$.
Gluskin~\cite{glu} showed that there are centrally-symmetric convex bodies with $$ d_{BM}(K_1, K_2) \geq c n. $$ We thus learn from Theorem \ref{thm0} that the
{\it Banach-Mazur diameter} in the centrally-symmetric case is not that different from the non-symmetric one. Gordon, Litvak, Meyer and Pajor \cite{GLMP} showed that
for any two convex bodies $K_1, K_2 \subseteq \RR^n$, when one of them is centrally-symmetric, we have
$$ d_{BM}(K_1, K_2) \leq n. $$  This follows from the fact, proven in \cite{GLMP},
that  for any two convex bodies $K_1, K_2 \subseteq \RR^n$ there exist invertible affine maps $T_1,T_2: \RR^n \rightarrow \RR^n$
such that $\tilde{K}_1 = T_1(K_1)$ and $\tilde{K}_2 = T_2(K_2)$ satisfy
$$ \tilde{K}_1 \subseteq \tilde{K}_2 \subseteq -n \tilde{K}_1. $$
Unlike the centrally-symmetric case, it is not known whether there exists a convex body $K_0 \subseteq \RR^n$
with $$ d_{BM}(K_0, T) \leq C \sqrt{n} \log^{\alpha} n $$
for all convex bodies $T \subseteq \RR^n$. If such a convex body $K_0 \subseteq \RR^n$ exists, then it would serve as an {\it approximate center}
of the non-symmetric Banach-Mazur compactum.

\medskip Unlike prior work, the invertible linear map that realizes the Banach-Mazur distance in Theorem \ref{thm0} is related to
the isotropic positions of $K_1$ and $K_2$. A random vector $X = (X_1,\ldots,X_n)$ in $\RR^n$ with finite second moments is {\it isotropic} if
 $$ \EE X_i = 0 \qquad \text{and} \qquad \EE X_i X_j = \delta_{ij} \qquad \qquad \qquad (i,j=1,\ldots,n) $$
 where $\delta_{ij}$ is the Kronecker delta. For any random vector $X$ in $\RR^n$ with $\EE |X|^2 < \infty$ that is not supported
 in a hyperplane, there exists an invertible affine  map $T: \RR^n \rightarrow \RR^n$ such that $T(X)$ is isotropic.
We say that a convex body $K \subseteq \RR^n$ is in isotropic position  if the random vector that is distributed uniformly in $K$ is isotropic.
The isotropic position is uniquely determined up to rotation. We refer to the Haar probability measure on the orthogonal group $O(n)$ as the uniform probability measure
on $O(n)$. Theorem \ref{thm0} follows from:

\begin{theorem} Let $K_1, K_2 \subseteq \RR^n$ be two convex bodies in isotropic position,
and let $U \in O(n)$ be a random orthogonal transformation, distributed uniformly in $O(n)$. Denote
$$ \lambda = C \sqrt{n} \log^{\alpha} n. $$
Then, with probability of at least $9/10$,
$$ U(K_1) \subseteq  \lambda  K_2.  $$
Here,  $C, \alpha > 0$ are universal constants. Our proof yields $\alpha \leq 2$.
\label{thm00}
\end{theorem}

Theorem \ref{thm00} shows
that the upper bound (\ref{eq_1042}) for the Banach–Mazur distance is attained when $K_1$ and $K_2$ are in a ``random isotropic position'', that is,
when each body is put in isotropic position and then independently rotated by a random rotation, uniformly distributed in $O(n)$.
The isotropic position is useful not only for controlling the Banach-Mazur distance, but also for obtaining estimates related to partial containment. Indeed, after a polylogarithmic dilation, an isotropic convex body contains almost all of the volume of any other isotropic convex body:

\begin{theorem} Let $K, T \subseteq \RR^n$ be convex bodies in isotropic position.
Then for $$ \lambda = C \log^{\alpha} n $$
we have
\begin{equation}  \Vol_n(\lambda K \cap T) \geq \frac{99}{100} \cdot \Vol_n(T), \label{eq_2111} \end{equation}
where $C, \alpha > 0$ are universal constants. Our proof yields $\alpha \leq 2$.
\label{thm_2118}
\end{theorem}

The number $99/100$ on the right-hand side of (\ref{eq_2111}) may be replaced by any number $1-\varepsilon \in (0,1)$,
at the expense of modifying the value of the universal constant $C$ from Theorem \ref{thm_2118} to $C_\varepsilon = C'\log(2/\eps)$, and using the exponent $\alpha'=5/2$ if both convex
bodies are non-symmetric. See the proof of Theorem \ref{thm_2118} for details.

\medskip
The {\it partial containment distance}
between two convex bodies $K_1, K_2 \subseteq \RR^n$, denoted by $d_{PC}(K_1, K_2)$,
 is
the infimum over all $\lambda \geq 1$ for which there exist an invertible
linear map $T: \RR^n \rightarrow \RR^n$ and vectors $x,y \in \RR^n$
such that $\tilde{K}_2 = y + K_2$ and $\tilde{K}_1 = T(x + K_1)$ satisfy
$$ \Vol_n(\sqrt{\lambda} \tilde{K}_2 \cap \tilde{K}_1) \geq \frac{99}{100} \cdot \Vol_n(\tilde{K}_1)
\qquad \text{and} \qquad \Vol_n(\sqrt{\lambda} \tilde{K}_1 \cap \tilde{K}_2) \geq \frac{99}{100} \cdot \Vol_n(\tilde{K}_2). $$
In other words, when scaling $\tilde{K}_2$ by a factor of $\sqrt{\lambda}$ it contains $99\%$  of $\tilde{K}_1$,
and when scaling $\tilde{K}_1$ by the same factor it contains $99\%$  of $\tilde{K}_2$.

\begin{corollary}
For any convex bodies $K_1, K_2 \subseteq \RR^n$,
$$ d_{PC}(K_1, K_2) \leq C \log^{\alpha} n, $$
where $C, \alpha > 0$ are universal constants. Our proof yields $\alpha \leq 4$. \label{cor_2109}
\end{corollary}

In the example where $K_1 = [-1,1]^n$ and $K_2 = B^n$, we have
\begin{equation}  d_{PC}(K_1, K_2) \geq c \sqrt{ \log n } \label{eq_950} \end{equation}
for a universal constant $c > 0$. Hence Corollary \ref{cor_2109} is optimal up to the values of $C$ and $\alpha$. See Lemma \ref{lem_1079} for a proof of this fact.

\subsection{Gaussian comparison}

The analytic input behind the geometric estimates described above is a two-sided Gaussian
comparison theorem for gauges of isotropic log-concave random vectors.

\medskip 
Recall that a probability density $p$ in $\RR^n$ is log-concave if
 the set $\{ x \in \RR^n ; \, p(x) > 0 \}$ is convex and the function $\log (1/p)$ is convex on this set.
  We say that an absolutely continuous probability measure (or a random vector with a density) on
  $\RR^n$  is log-concave if its density is log-concave.
 The uniform probability measure on a convex body is log-concave, as is the standard Gaussian
 measure in $\RR^n$.  We say that the random vector $X$ in $\RR^n$ satisfies a Poincar\'e inequality with constant $A$ if
for all locally-Lipschitz functions $f: \RR^n \rightarrow \RR$,
\begin{equation} \var f(X) \leq A \cdot \EE |\nabla f(X)|^2. \label{eq_723} \end{equation}
The minimal $A$ for which (\ref{eq_723}) holds  is called the {\it Poincar\'e constant} of $X$, and it is denoted by $C_P(X)$.
By testing linear functions $f$ we see that $C_P(X) \geq 1$ whenever $X$ is isotropic.
We define the Kannan-Lov\'asz-Simonovits (KLS) constant $\psi_n$ as
$$ \psi_n = \sup_X \sqrt{C_P(X)}, $$
where the supremum runs over all isotropic, log-concave random vectors $X$ in $\RR^n$.
Clearly $\psi_n \geq 1$. The KLS conjecture \cite{KLS} suggests that $\psi_n \leq C$ for a universal constant $C > 0$.
It is currently known that
\begin{equation}  \psi_n \leq C \sqrt{\log n}, \label{eq_1100} \end{equation}
 as proven in \cite{k_sqrt}.

 \medskip In his recent lecture,  Giannopoulos \cite{gian}
 suggested revisiting the $M M^*$ problem in isotropic position.
 There are several well-known positions (i.e., affine images) of a convex body in $\RR^n$, such as the John position, the $\ell$-position, the $M$-position (named after Milman), and the isotropic position. The resolution of the slicing problem
 in Klartag and Lehec \cite{KL}, building upon ideas of Guan \cite{guan},  implies that the isotropic position is an $M$-position. See also Bizeul \cite{biz} for an alternative proof. Moreover, the isotropic position is well-known to imply the bound (\ref{eq_308}), exactly like the John position. Indeed,  if $K \subseteq \RR^n$ is in isotropic position, then,
 \begin{equation}  \sqrt \frac{n+2}{n} B^n \subseteq K \subseteq \sqrt{n(n+2)} B^n, \label{eq_1935} \end{equation}
 as proven in \cite{KLS}.  Giannopoulos put forth the question of whether the isotropic position is as good as the $\ell$-position for bounding $M M^*$. We show that this is indeed the case, up to polylogarithmic factors.
 The new ingredient in the proofs of our main results is the following:

\begin{theorem} Let $X$ be an isotropic, log-concave random vector in $\RR^n$ and
let $G$ be a standard Gaussian random vector in $\RR^n$. Then for any norm $\| \cdot \|$ on $\RR^n$, denoting $Q_n = \psi_n \sqrt{\log n}$,
\begin{equation}  \frac{c}{Q_n } \cdot \EE \| G \| \leq \EE \| X \| \leq C Q_n \cdot \EE \| G \|,
\label{eq_2026} \end{equation}
where $C,c > 0$ are universal constants. In fact, (\ref{eq_2026}) also  holds when
$\| \cdot \|$ is a ``gauge function'' on $\RR^n$, that is, when $\| \cdot \|$ is a real-valued, non-negative, convex function in $\RR^n$ satisfying $\| \lambda x \| = \lambda \| x \|$ for all $x \in \RR^n$ and $\lambda \geq 0$. \label{thm1}
\end{theorem}

The inequality on the right-hand side of (\ref{eq_2026}) is due to Eldan and Lehec \cite{eldan_lehec}.
Our contribution is the proof of the inequality on the left-hand side.  Theorem \ref{thm11} below is a formally stronger version of Theorem \ref{thm1},
where the KLS constant $\psi_n$ is replaced by a certain coefficient $\kappa_n \leq 2 \psi_n$
related to the tensor of $3^{rd}$ moments of isotropic, log-concave random vectors in $\RR^n$.
We write that $$ \mu_1 \prec \mu_2 $$ in the gauge order if $$ \int_{\RR^n} \| x \| d \mu_1(x) \leq \int_{\RR^n} \| x \| d \mu_2(x) $$ for any gauge function $\| \cdot \|$ on $\RR^n$ (e.g. a norm).
By using the best available bound (\ref{eq_1100}) for the KLS constant, Theorem \ref{thm1} admits the following:

\begin{corollary} Let $\mu$ be an isotropic, log-concave probability measure on $\RR^n$. Then in the gauge order,
$$ \gamma_{c / \log^2 n} \prec \mu \prec \gamma_{C \log^2 n}, $$
where $\gamma_t$ is a Gaussian probability measure of mean zero and covariance $t \cdot \id$ in $\RR^n$,
and where $C, c > 0$ are universal constants. \label{cor_1403}
\end{corollary}

If the KLS conjecture holds true, then the $\log^2 n$ factors in Corollary \ref{cor_1403} are improved to $\log n$  which would be optimal; see Section~\ref{rem_1102} below.

\medskip
Corollary \ref{cor_1403} may be compared with the work of Fernique and Talagrand \cite{tal_chaining} on sub-Gaussian random vectors in $\RR^n$,
and with the work of Bobkov and Nazarov \cite{BN} and Lata\l{}a \cite{Lat} on unconditional, log-concave distributions.
We say that a probability measure $\mu$ in $\RR^n$ is sub-Gaussian with parameter $\sigma > 0$ if its barycenter lies at the origin and for any $\theta \in S^{n-1}$,
$$ \int_{\RR^n} \exp \left( \frac{\langle x, \theta \rangle^2}{\sigma^2} \right) d \mu(x) \leq 2. $$
The Fernique-Talagrand theory implies that in such a case, we have the one-sided comparison
$$ \mu \prec \gamma_{C \sigma^2},  $$
where $C > 0$ is a universal constant. It is customary to refer to a probability measure $\mu$ in $\RR^n$ as {\it unconditional} if it is invariant
under coordinate reflections in $\RR^n$. Write $\nu_{\sigma}$ for the probability measure in $\RR^n$ with density $\sigma^{-n} \exp(-\sum_{i=1}^n 2 |x_i| / \sigma)$.
It is shown in Lata\l{}a \cite{Lat} that if $\mu$ is an unconditional, isotropic, log-concave
probability measure on $\RR^n$, then we have the one-sided comparison
$$ \mu \prec \nu_{C} $$
for a universal constant $C > 0$.
While in general we cannot upgrade the gauge-order comparison of Corollary \ref{cor_1403}
to a convex-order comparison, there are large classes of convex functions for which comparison principles
hold true; see Remark \ref{rem_1814}  below.

\medskip For a convex body $K \subseteq \RR^n$ containing the origin in its interior we
consider the Minkowski functional $\| x \|_K = \inf \left \{ \lambda \geq 0 \, ; \, x \in \lambda K \right \}$, and define
$$ M(K) = \int_{S^{n-1}} \| x \|_K d \sigma_{n-1}(x) $$
where $S^{n-1} = \left \{ x \in \RR^n \, ; \, |x| =1 \right \}$ and $\sigma_{n-1}$ is the uniform probability measure
on $S^{n-1}$. The polar body to $K$ is of course $K^{\circ} = \left \{ x \in \RR^n \, ; \, \forall y \in K, x \cdot y \leq 1 \right \}$ and we abbreviate
$$ M^*(K) = M(K^{\circ}), $$
which is half of the {\it mean width} of $K$.

\begin{corollary} Let $K \subseteq \RR^n$ be a convex body in isotropic position. Then,
\begin{equation}  M(K) \leq C \frac{\psi_n \sqrt{\log n}}{\sqrt{n}}, \label{eq_1042_} \end{equation}
and
\begin{equation}  M^*(K) \leq C \log^2 n \cdot \sqrt{n}, \label{eq_1115} \end{equation}
where $C > 0$ is a universal constant. \label{cor_1617}
\end{corollary}

Inequality (\ref{eq_1115}) is due to E. Milman \cite{e_milman}. Corollary \ref{cor_1617} implies that
when $K \subseteq \RR^n$ is a convex body in isotropic position,
\begin{equation} M(K) M^*(K) \leq C \log^\alpha n
\label{eq_1003} \end{equation}
for $\alpha \leq 3$ and a universal constant $C > 0$.

\medskip The problem of finding a position where the $M M^*$ product
grows only logarithmically with the dimension in the non-symmetric case has been studied at least since the 1990s.
The bound in the centrally-symmetric case is a cornerstone
of Banach space geometry; see Pisier \cite{pisier_book} for its proof, history and various applications.

\medskip In the non-symmetric case, it was proven in Banaszczyk, Litvak, Pajor and Szarek \cite{BLPS}
that in an appropriate position, the $M(K) M^*(K)$ product is at most $C \sqrt{ d_{BM}(K, B^n) }$.
In Rudelson \cite{rudelson} it is proven
that in an appropriate position, this product is at most $C n^{1/3} \log^{\alpha} n$.
When combined with the results of Banaszczyk \cite{ban} and Banaszczyk, Litvak, Pajor and Szarek \cite{BLPS},  the bound (\ref{eq_1003}) yields
an alternative proof of a recent flatness bound due to Reis and Rothvoss \cite{RR}.

\medskip The question of bounding $M(K)$ in the isotropic position, answered by Corollary \ref{cor_1617}
up to polylogarithmic factors, was studied by
Giannopoulos, Stavrakakis, Tsolomitis and Vritsiou \cite{GSTV}
and by Giannopoulos and E. Milman \cite{GM2}. An inequality analogous to (\ref{eq_1042_}) with $C n^{-1/8} \log^\alpha n$ on the right-hand side
was proven in \cite{GM2} under the assumption that $K \subseteq \RR^n$
is centrally-symmetric, improving upon earlier estimates from \cite{GSTV}.
A generalization to the non-symmetric case was obtained by Vritsiou \cite{V}.
It is shown in Skarmogiannis \cite{skar} that up to polylogarithmic factors, Corollary \ref{cor_1617} completes the  proof of a conjecture attributed to V. Milman on multi-integral norms; see  \cite{skar} for background and for more information.

\medskip The remainder of this paper is organized as follows. Section \ref{sec_background}
provides  background on heat flow methods and stochastic localization in convex geometry.
Theorem \ref{thm1} is proven in Section \ref{sec3}. In Section \ref{sec4}, Theorem \ref{thm1} is combined with Chevet's inequality to establish the other theorems and corollaries stated above,
as well as applications to the first Dirichlet eigenvalue of the Laplacian and to linear symplectic geometry. These applications
illustrate the recurring role of the isotropic position in controlling various geometric parameters of
convex bodies, up to at most polylogarithmic factors in the dimension.

\medskip In this paper, we write $| x | = \sqrt{\sum_{i=1}^n x_i^2}$ for the Euclidean
norm of $x = (x_1,\ldots,x_n) \in \RR^n$, with $x \cdot y = \langle x, y \rangle = \sum_{i=1}^n x_i y_i$
being the standard scalar product between $x,y \in \RR^n$. For $x,y \in \RR^n$ we write $x \otimes y = (x_i y_j)_{i,j=1,\ldots,n} \in \RR^{n \times n}$.
The transpose of a matrix $A$ is denoted by $A^*$ and its trace is denoted by $\tr[A]$. We write $\| A \|_{op} = \sup_{x \neq 0} |A x| / |x|$ for the operator norm of a matrix $A \in \RR^{n \times n}$.
The support of a Borel measure $\mu$, denoted by $\supp(\mu)$, is the complement of the union of all open sets of zero $\mu$-measure.
The Hessian of a smooth function $u: \RR^n \rightarrow \RR$ at the point $x \in \RR^n$ is denoted by $\nabla^2 u(x) \in \RR^{n \times n}$.
We write $C, c, C', \tilde{c}, \bar{C}$ etc. to denote various positive universal constants whose value may change from one line to the next. We write that $A \sim B$ if $c A \leq B \leq C A$ for universal constants $c, C > 0$, and that $A \lesssim B$ if $A \leq C B$.

\medskip {\it Acknowledgements.} We are grateful to Apostolos Giannopoulos, Sasha Litvak, Vitali Milman, Yaron Ostrover and Grigoris Paouris for stimulating discussions, and
to Larry Guth for valuable correspondence on the partial containment distance. This work was carried out during our stay at the ETH Institute for Theoretical Studies, whose generous hospitality is gratefully acknowledged. We also gratefully acknowledge the support of the Israel Science Foundation (ISF).

\section{Background on stochastic localization}
\label{sec_background}

Let $\mu$ be a compactly-supported probability measure on $\RR^n$ with density $p$. In this section
we study the evolution of the measure $\mu$ under the heat flow in $\RR^n$ from the point of view of stochastic localization. To this end, for $t \geq 0$ and $\theta\in\RR^n$ we define a
probability density
$$p_{t,\theta}(x) = e^{\theta\cdot x - t|x|^2/2 - \Lambda_t(\theta)}p(x),$$
where
\begin{equation}
\Lambda_t(\theta) = \log \int_{\RR^n}  e^{\theta\cdot x -t|x|^2 / 2} p(x) dx. \label{eq_144} \end{equation}
We further denote by $a_t(\theta) \in \RR^n$ and $A_t(\theta) \in \RR^{n \times n}$ the barycenter and covariance matrix of $p_{t,\theta}$, namely
\begin{equation} a_t(\theta) = \nabla \Lambda_t(\theta) = \int_{\RR^n} x p_{t,\theta}(x)dx \label{eq_2041} \end{equation}
and
\begin{equation} A_t(\theta) = \nabla^2 \Lambda_t(\theta)  = \int_{\RR^n} (x \otimes x) p_{t, \theta}(x) dx - a_t(\theta) \otimes a_t(\theta).
\label{eq_2052} \end{equation}
Since $\mu$ is compactly-supported, the functions $a_t: \RR^n \rightarrow \RR^{n}$ and $A_t: \RR^n \rightarrow \RR^{n \times n}$
are bounded and Lipschitz, uniformly in $t \geq 0$. Indeed, any partial derivative of first order of $a_t$ or of $A_t$
is a combination of first, second and third moments of the probability density $p_{t, \theta}$ whose support is contained in the compact set $\supp(\mu)$,
see e.g.  \cite[Section 2]{KL_thin}.

\medskip
Let $(B_t)_{t \geq 0}$ be a standard Brownian motion in $\RR^n$ with $B_0 = 0$.
The {\it tilt process} $(\theta_t)_{t\geq0}$ is defined as the solution to the stochastic differential equation
\begin{equation}\label{eq_tilt}
    d\theta_t = dB_t + a_t(\theta_t)\, dt;\qquad \theta_0=0.
\end{equation}
A strong solution to this stochastic differential equation exists and is unique (e.g., {\O}ksendal \cite[Chapter 5]{oksendal}).
In fact, as explained in \cite[Section 2]{KL_thin}, the integral equation reformulation
of (\ref{eq_tilt}) admits a unique solution for any continuous path $(B_t)_{t \geq 0}$,
by classical existence and uniqueness results for ordinary differential equations.
We define the {\it stochastic localization process} $(\mu_t)_{t \geq 0}$ starting from $\mu$, as the measure-valued stochastic process with density
\[
    p_t := p_{t,\theta_t}.
\]
This process and variants thereof were introduced by Eldan \cite{eldan_phd, eldan2013thin} as tools for proving geometric inequalities in high dimensions. The process appeared earlier in other contexts such as non-linear filtering (e.g., Chigansky \cite[Example 6.15]{chiganski}), and it is closely related to pathwise analysis of the heat flow in $\RR^n$.
For example, it is explained in Klartag and Putterman \cite{klartag2023spectral} that if $X$ is a random vector with law $\mu$
that is independent of $(B_t)_{t \geq 0}$, then
$$(\theta_t)_{t\geq0} \sim (tX + B_t)_{t\geq0}, $$
i.e., the tilt process $(\theta_t)_{t\geq0}$ coincides in law with the stochastic process $(tX + B_t)_{t\geq0}$.
Furthermore, for each $t>0$, the random probability measure $\mu_t$ has the same distribution as the conditional law
\begin{equation}
\mathcal{L}(X \mid tX + B_t) \;=\; \mathcal{L}(X \mid X + B_t/t).
\label{eq_1153} \end{equation}
where $\mathcal{L}(Y)$ denotes the law of a random vector $Y$.
The barycenter and covariance of
the random probability measure $\mu_t$ are denoted by  $$ a_t = a_t(\theta_t) \in \RR^n $$ and $$ A_t = A_t(\theta_t) \in \RR^{n \times n}. $$
We refer to $(A_t)_{t \geq 0}$ as the covariance process of the stochastic localization process starting at $\mu$.
It follows from \eqref{eq_tilt} and an It\^o computation that the density $p_t$ satisfies the evolution equation
\begin{equation}\label{eq_density}
    dp_t(x) = p_t(x)(x-a_t)\cdot dB_t \qquad \qquad \text{for} \ x\in\RR^n.
\end{equation}
From \eqref{eq_density} we deduce that $(\mu_t)_{t\geq0}$ is a martingale, in the sense that for any bounded continuous test function $f: \RR^n \rightarrow \RR$,
 the process
\begin{equation} M_t = \int_{\RR^n} f d\mu_t \label{eq_1135} \end{equation}
is a martingale. Consequently, since $\mu_0 = \mu$, the measure $\mu$ can be written as an average of the random probability measures $\mu_t$
\begin{equation}\label{eq_martingale}
    \mu = \EE \mu_t  \qquad \qquad \text{for} \  t \geq 0.
\end{equation}
We see from (\ref{eq_density}) that the stochastic differential equation satisfied by the martingale in (\ref{eq_1135}) is
\begin{equation}\label{eq_172}
d \left[ \int_{\RR^n} f d\mu_t \right] = \left(\int_{\RR^n}f(x)(x-a_t)d\mu_t\right)\cdot dB_t = \cov_{\mu_t}(x,f)\cdot dB_t,
\end{equation}
where $\cov_{\nu}(f,g) = \int fg d \nu - \int f d\nu \int g d \nu$. Formula \eqref{eq_172} may be used to compute the dynamics of $a_t$, namely,
\begin{equation}\label{eq_barycenter}
    da_t = \cov(\mu_t) dB_t = A_tdB_t,
\end{equation}
where we recall that $\cov(\nu) = \int [x \otimes x] d \nu - \int x d \nu \otimes \int x d \nu \in \RR^{n \times n}$ is the covariance
matrix of the compactly-supported probability measure $\nu$ in $\RR^n$.

\medskip
Observe that the density $p_{t} = p_{t,\theta_t}$ is proportional to the product of a deterministic factor
$p(x) e^{-t |x|^2/2}$ and a random exponential tilt $e^{\theta_t \cdot x}$.
From now on, let us make the assumption that the measure $\mu$ is {\it log-concave}.
This implies that almost surely
\begin{equation}\label{eq176}
    -\nabla^2\log p_t(x) \geq t \cdot \id, \qquad \textrm{for } \mu_t\textrm{-almost all } \ x \in \RR^n.
\end{equation}
A uniform bound on the Hessian of the log-density in the form of \eqref{eq176} is well known to imply a Poincar\'e inequality with constant $1/t$.
Indeed, the log-concave Lichnerowicz inequality (e.g. \cite{k_sqrt}) states that almost surely
$$ C_P(\mu_t) \leq \frac{1}{t}.$$
By substituting linear functions in the Poincar\'e inequality (\ref{eq_723}), we deduce that almost surely
\begin{equation}\label{eq_182}
    A_t \leq \frac{1}{t} \cdot \id.
\end{equation}
The following lemma is well-known. It essentially follows from (\ref{eq_1153}) since
$X + B_t/t$ tends to $X$ as $t \to \infty$, and hence $\mu_{t, t X + B_t}$ tends to $\delta_X$, where $\mu_{t, \theta}$ is the measure whose density is $p_{t, \theta}$
and where $\delta_x$ is a Dirac delta measure at the point $x \in \RR^n$.
For completeness we provide a proof that relies
on our standing assumption that $\mu$ is log-concave.

\begin{lemma}\label{lem190} The random vector $a_{\infty} = \lim_{t \to \infty} a_t$ is well-defined and has law $\mu$.
Furthermore, almost surely the random probability measure $\mu_t$ converges weakly
as $t \to \infty$ to a Dirac delta measure $\mu_\infty$ located at~$a_\infty$.
\end{lemma}
\begin{proof} Since $\mu$ is compactly-supported, the process $(a_t)_{t \geq 0}$ is a bounded martingale
as we see for instance from (\ref{eq_barycenter}). By Doob's martingale convergence theorem (e.g. \cite[Theorem 3.19]{legall}), $a_t$ converges almost surely to some random vector $a_\infty$ as $t \to \infty$. Now,
    \begin{align*}
        W_2(\mu_t, \delta_{a_\infty}) &\leq W_2(\mu_t, \delta_{a_t}) + W_2(\delta_{a_t},\delta_{a_\infty})\\
        & = \Tr(A_t)^{1/2} + |a_t - a_\infty| \xrightarrow[t \to \infty]{} 0 \ \textrm{a.s.}
    \end{align*}
    where $W_2$ is the Wasserstein $2$-distance and we used the deterministic bound \eqref{eq_182} on the covariance matrix $A_t$. We have thus established that $\mu_t$ converges to $$ \mu_{\infty} := \delta_{a_\infty} $$ in the Wasserstein $2$-distance, almost surely. This implies in particular weak convergence. Thus for any bounded, continuous test function $f: \RR^n \rightarrow \RR$,
    \begin{equation} \EE f(a_\infty) = \EE \int_{\RR^n} f d \mu_{\infty} = \lim_{t\rightarrow +\infty} \, \EE \int_{\RR^n} f d \mu_t = \int_{\RR^n} f d \mu, \label{eq_1157} \end{equation}
    where we used (\ref{eq_martingale}) in the last passage. From (\ref{eq_1157}) we see that  $a_\infty$ has law $\mu$.
\end{proof}

\begin{corollary}\label{lem214} Let $X$ be a log-concave random vector in $\RR^n$ with law $\mu$. Assume that $\EE X = 0$. Let $A_t$ be the covariance process of the stochastic localization process starting at $\mu$. Then the following equality holds in law:
    $$X \sim \int_{0}^\infty A_tdB_t.$$
\end{corollary}

\begin{proof} Recall formula (\ref{eq_barycenter}) for the dynamics of $(a_t)$.
Thus, according to Lemma \ref{lem190}, we have
$$X\sim a_\infty = a_0 + \int_0^\infty A_tdB_t = \int_0^\infty A_tdB_t, $$
since $a_0 = \EE X = 0$.
\end{proof}

In order to make substantial use of the representation given by Corollary \ref{lem214}, we need to analyze the covariance process $(A_t)_{t \geq 0}$. From now on we assume that $\mu$ is log-concave and {\it isotropic}, namely
$$A_0 = \cov\!(\mu) = \id.$$
Write $B_t = (B_{1,t},\ldots,B_{n,t}) \in \RR^n$, and recall that $B_{1,t},\ldots, B_{n,t}$ are independent standard Brownian motions in $\RR$ starting from zero. Using \eqref{eq_172}
and the It\^o formula, we compute the dynamics of $A_t$ as follows (see e.g., \cite{KL_bulletin}):
\begin{equation} \label{eq_1204}
    dA_t = \sum_{i=1}^n H_{i,t} dB_{i,t} -A_t^2dt
\end{equation}
where $H_t = (H_{1,t},\ldots,H_{n,t})$ is the tensor of third moments of $\mu_t$, i.e.
\begin{equation}
    H_{i,t} = \int_{\RR^n} (x-a_t)_i \left[ (x - a_t) \otimes (x - a_t) \right] d\mu_t \in \RR^{n \times n},
    \label{eq_1234}
\end{equation}
and $(x-a_t)_i$ is the $i^{th}$ coordinate of the vector $x - a_t \in \RR^n$.

\medskip Write $\RR^{n \times n}_{sym}$ for the vector space of real, symmetric $n \times n$ matrices.
Given a smooth functional $f: \RR^{n \times n}_{sym} \rightarrow \RR$, from (\ref{eq_1204}) and the It\^o  formula,
\begin{equation}\label{eq155}
    df(A_t) = \nabla f(A_t)\cdot\sum_{i=1}^nH_{i,t}dB_{i,t} -\nabla f(A_t)\cdot A_t^2 dt + \frac{1}{2}\sum_{i=1}^n \nabla^2f(A_t)(H_{i,t}\,,\,H_{i,t}) \, dt.
\end{equation}
We need to analyze the various terms in (\ref{eq155}).
As in Eldan \cite{eldan2013thin} we define
$$ \kappa_n = \sup_X \sqrt{ \| \EE X_1 (X \otimes X) \|^2_{HS} } $$
where the supremum runs over all isotropic, log-concave random vectors $X = (X_1,\ldots,X_n)$ in $\RR^n$,
and where $$ \| M \|_{HS} = \sqrt{\sum_{i,j=1}^n M_{ij}^2} $$ is the Hilbert-Schmidt norm of
the matrix $M = (M_{ij})_{i,j=1,\ldots,n} \in \RR^{n \times n}$.
Clearly for any isotropic log-concave random vector $X$ in $\RR^n$ and for any $\theta \in S^{n-1}$,
$$  \| \EE \langle X, \theta \rangle (X \otimes X) \|_{HS}^2  \leq \kappa_n^2. $$
A  scaling argument (see \cite[Lemma 5.1]{klartag2022bourgain}) shows that for any centered, log-concave
random vector $X$ in $\RR^n$ and for any $\theta \in S^{n-1}$,
\begin{equation}
 \| \EE \langle X, \theta \rangle (X \otimes X) \|_{HS}^2  \leq \| \cov(X) \|_{op}^{3} \kappa_n^2.
\label{eq_1945} \end{equation}
Recall the KLS constant $\psi_n$ from Section \ref{intro}. The quantity $\kappa_n$  satisfies
\begin{equation}
2 \leq \kappa_n  \leq 2 \psi_n.
\label{eq_1737}
\end{equation}
Indeed, the inequality on the left-hand side in (\ref{eq_1737}) is proven
by considering the case where $1 + X_1$ is an exponential random variable with parameter $1$,
so that $\EE X_1^3 = 2$ and $\kappa_n \geq 2$. As for the inequality on the right-hand side of (\ref{eq_1737}),
for any isotropic, log-concave random vector $X$ in $\RR^n$,
denoting $M = \EE X_1 (X \otimes X)$ and $g(x) = \langle M x, x \rangle$ we have
\begin{align*} \| M \|_{HS}^2 & = \EE X_1 \langle MX, X \rangle \leq \sqrt{ \EE X_1^2} \cdot \sqrt{ Var( g(X) ) }
\leq \sqrt{C_P(X) \cdot \EE |\nabla g(X)|^2 } \\ & = 2  \sqrt{ C_P(X) \cdot \EE |MX|^2} = 2 \sqrt{ C_P(X) }  \| M \|_{HS}
\leq 2 \psi_n \| M \|_{HS}.
\end{align*}
This implies (\ref{eq_1737}).
 The following proposition is a routine variant of well-known estimates in the theory
 of stochastic localization, going back to Eldan \cite{eldan2013thin} in a slightly different context.
 In our precise context, a proof of the right-hand side bound in (\ref{eq_1842}) appears in
 Klartag and Lehec \cite{klartag2022bourgain} and also in Bizeul \cite{bizeul2024measures}, while a version
 of the left-hand side inequality assuming a certain polynomial bound on $\psi_n$
 appears in Lee and Vempala \cite{lee2024eldan}.

\begin{proposition}\label{lem161}
    Let $T = (\tilde{C}\kappa_n^2\log n)^{-1}$. Then, with probability of at least $1-C \exp\left(-c / T\right)$,
\begin{equation} \frac{\id}{2}  \leq A_t \leq 2\id \qqquad \text{for all} \ 0\leq t\leq T, \label{eq_1842} \end{equation}
   where $\tilde{C}, C,c>0$ are universal constants.
\end{proposition}

The proof of Proposition \ref{lem161}
is provided in the Appendix below. It is based on analysis of formula (\ref{eq155}) along the lines of \cite{klartag2022bourgain}.

\section{Comparing log-concave integrals with Gaussian integrals}
\label{sec3}

An observation attributed to Maurey by Eldan and Lehec \cite[Proposition 8]{eldan_lehec}
is that the endpoint value of a martingale in $\RR^n$ whose quadratic variation is
uniformly bounded, is the average of two possibly-dependent Gaussians:

\begin{proposition}[Maurey]
Let $T, r > 0$, and let $(M_t)_{t \geq 0}$ be a continuous martingale in $\RR^n$ with $M_0 = 0$. Assume
that its quadratic variation process satisfies,
almost surely,
$$ [M]_T \leq r \cdot \id. $$
Then there exist two standard Gaussian random vectors $Z_1, Z_2$ in $\RR^n$, possibly dependent, such that
\begin{equation}  M_T = \sqrt{r} \cdot \frac{Z_1 + Z_2}{2}. \label{eq_1142} \end{equation}
\label{prop_1155}
\end{proposition}

\begin{proof} As in the proof of \cite[Proposition 8]{eldan_lehec}, let $G$ be a standard Gaussian random vector in $\RR^n$
independent of $(M_t)_{t \geq 0}$. Set
$$ Y_{\pm} = M_T \pm (r \cdot \id - [M]_T)^{1/2} G, $$
where $A^{1/2}$ is the symmetric, positive semi-definite square root of the symmetric, positive semi-definite matrix $A$.
The random vectors $Z_1 = Y_+ / \sqrt{r}$ and $Z_2 = Y_- / \sqrt{r}$  clearly satisfy (\ref{eq_1142}),
and all that remains is to show that $Y_{\pm}$ is a Gaussian random vector of mean zero and covariance $r \cdot \id$.
To this end, we compute that for any fixed $x \in \RR^n$,
\begin{align} \nonumber \EE e^{i \langle x, Y_{\pm} \rangle }
& = \EE_{(M_t)} \EE_G \left[ e^{i \langle x, M_T \pm (r \cdot \id - [M]_T)^{1/2} G \rangle} \, |  \, (M_t)_{0 \leq t \leq T} \right]
\\ & = \EE_{(M_t)} \left[ e^{i \langle x, M_T \rangle - |(r \cdot \id - [M]_T)^{1/2} x|^2/2 } \right] = e^{-r |x|^2/2} \cdot \EE \left[ e^{i \langle x, M_T \rangle + \langle [M]_T x, x \rangle / 2}
\right]. \label{eq_1151}
\end{align}
By Itô's formula, the process $$ \left( e^{i \langle x, M_t \rangle + \tfrac{1}{2}\langle [M]_t x, x \rangle} \right)_{0 \leq t \leq T} $$ is a complex-valued local martingale. Moreover, since $[M]_t \leq [M]_T \leq r\id$, its modulus is bounded and therefore it is a true martingale
. Hence the last expectation in (\ref{eq_1151}) equals $$ e^{i x \cdot M_0} = 1. $$ We have thus shown that the characteristic
function of $Y_\pm$ is that of a Gaussian of mean zero and covariance $r \cdot \id$.
Since the characteristic function determines the distribution, the proof is complete.
\end{proof}

The next proposition
is concerned with the opposite situation, when the derivative of the quadratic variation of the martingale is at least
$r \cdot \id$. Recall that $(B_t)_{t \geq 0}$ is a standard Brownian motion in $\RR^n$ with $B_0 = 0$.

\begin{proposition}
Let $r, T > 0$ and let $(M_t)_{t \geq 0}$ be a continuous martingale in $\RR^n$ of the form
$$ M_t = \int_0^t \Sigma_s d B_s \qquad \qquad \qquad (0 \leq t \leq T) $$
where $(\Sigma_s)_{s \geq 0}$ is a uniformly-bounded, continuous, matrix-valued stochastic process, adapted to the filtration
of the Brownian motion $(B_t)$, such that for any $t \in [0,T]$,
$$ \Sigma_t \Sigma_t^* \geq r \cdot \id \qquad \qquad \text{almost surely}. $$
Then for any  convex function $F: \RR^n \rightarrow \RR$ that does not grow faster than polynomially at infinity,
\begin{equation}  \EE F(M_T) \geq \EE F(\sqrt{r} B_T). \label{eq_145} \end{equation}
\label{cor_438}
\end{proposition}

\begin{proof} We may assume that $r = 1$, since otherwise we can divide  $M_t$ and $\Sigma_t$ by $\sqrt{r}$
and replace the convex function $F(x)$ by the convex function $F(\sqrt{r} x)$. For $t > 0$ and $x \in \RR^n$
consider the convolution
\begin{equation}  u_t (x) = (F * \gamma_t)(x) < \infty \label{eq_149} \end{equation}
where $\gamma_t(x) = (2 \pi t)^{-n/2} \exp(-|x|^2/(2t))$ is the density of a Gaussian random vector of mean zero and covariance $t \cdot \id$ in $\RR^n$.
We define by continuity $$ u_0(x) = F(x), $$ and observe that $u_t$ is convex for all $0 \leq t \leq T$. Moreover, by the growth assumptions on $F$ we are allowed to differentiate the integral defining the convolution in (\ref{eq_149}) under the integral sign, and conclude that for $t > 0$ the function $u_t$ is a smooth function
all of whose partial derivatives of all orders grow at most polynomially at infinity, and that we have the {\it heat equation}:
\begin{equation}  \partial_t u_t = \frac{\Delta u_t}{2}. \label{eq_321} \end{equation}
In order to prove (\ref{eq_145}), it suffices to show that
\begin{equation}
\EE u_T(M_0) \leq \EE u_0(M_T).
\label{eq_157} \end{equation}
Indeed, the right-hand side  of (\ref{eq_157}) equals $\EE F(M_T)$ while the left-hand side  equals $u_T(0) = \EE F(B_T)$.
In order to prove (\ref{eq_157}) we will show that the stochastic process
$$ u_{T-t}(M_t) \qquad \qquad \qquad (0 \leq t \leq T), $$
is a sub-martingale, and in particular its expectation is non-decreasing in $t$. Our growth assumptions on the convex function $F$ and the fact that $\Sigma_t$ is uniformly bounded
imply that $\EE \int_0^T |\nabla u_{T-t}(M_t)|^2 dt < \infty$. Hence the stochastic process
$$ \int_0^t \nabla u_{T-s}(M_s) \cdot \Sigma_s d B_s \qquad \qquad \qquad (0 \leq t \leq T) $$
is a martingale. By the It\^o formula and (\ref{eq_321}),
\begin{align*} d u_{T-t}(M_t) & = \partial_t u_{T-t}(M_t) dt + \nabla u_{T-t}(M_t) \cdot \Sigma_t d B_t + \frac{1}{2} \tr \left[ \nabla^2 u_{T-t}(M_t) \Sigma_t \Sigma_t^*  \right] dt \\ & = \nabla u_{T-t}(M_t) \cdot \Sigma_t d B_t + \frac{1}{2} \tr \left[ \nabla^2 u_{T-t}(M_t) (\Sigma_t \Sigma_t^* - \id) \right] dt.
\end{align*}
Since both symmetric matrices $\nabla^2 u_{T-t}(M_t)$ and $\Sigma_t \Sigma_t^* - \id$ are positive semi-definite, the stochastic process  $(u_{T-t}(M_t))_{0 \leq t \leq T}$ is the sum of a martingale and a non-decreasing term, hence a sub-martingale.
\end{proof}

Let $\mu$ be an isotropic, log-concave, compactly-supported probability measure on $\RR^n$.
Recall the processes $(\theta_t)_{t \geq 0}, (a_t)_{t \geq 0}$ and $(A_t)_{t \geq 0}$ described at the beginning of Section \ref{sec_background}. For a function $f: \RR \rightarrow \RR$ and a symmetric matrix
$M \in \RR^{n \times n}$ whose spectral decomposition is $$ M = \sum_{i=1}^n \lambda_i u_i \otimes u_i, $$
for an orthonormal basis $u_1,\ldots,u_n \in \RR^n$ and $\lambda_1,\ldots,\lambda_n \in \RR$,
we write
$$ f(M) = \sum_{i=1}^n f(\lambda_i) u_i \otimes u_i. $$
Consider the truncation function $$ f(u) = \min \{ \max \{u, 1/2 \}, 2 \}, $$
and let us denote
$$ \Sigma_t = f(A_t). $$
Then $(\Sigma_t)_{t \geq 0}$ is a continuous stochastic process in $\RR^{n \times n}_{sym}$,
adapted to the filtration of the Brownian motion $(B_t)_{t \geq 0}$, such that for any $t > 0$, almost surely,
\begin{equation}  \frac{\id}{2} \leq \Sigma_t \leq 2 \id. \label{eq_1726} \end{equation}
Consider the continuous martingale
\begin{equation} v_t = \int_0^t \Sigma_s d B_s \qquad \qquad \qquad (t \geq 0). \label{eq_1706} \end{equation}
Proposition~\ref{prop_1155} and Proposition~\ref{cor_438} imply the following corollary.

\begin{corollary} Let $F: \RR^n \rightarrow \RR$ be a convex function that does not grow faster than polynomially at infinity.
Then for any $t > 0$,
\begin{equation}  \EE F( B_t / 2) \leq \EE F(v_t) \leq \EE F( 2  B_t). \label{eq_1725} \end{equation}
Moreover, with probability of at least $1 - 1/n^{20}$, for all $0 \leq t \leq c_0 \kappa_n^{-2} / \log n$,
$$ a_t = v_t, $$
where $c_0 > 0$ is a universal constant.
\label{lem_2205}
\end{corollary}

\begin{proof} It follows from (\ref{eq_1726}) and (\ref{eq_1706}) that $[v]_t \leq 4 t \cdot \id$.
Hence by Proposition \ref{prop_1155},
$$ \EE F(v_t) = \EE F \left(\frac{2 \sqrt{t} Z_1 + 2 \sqrt{t} Z_2}{2} \right) \leq \EE \frac{F(2 \sqrt{t} Z_1) +
F(2 \sqrt{t} Z_2)}{2} = \EE F(2 B_t). $$
This proves the right-hand side inequality in (\ref{eq_1725}).
Thanks to (\ref{eq_1726}) and (\ref{eq_1706}) we may apply Proposition \ref{cor_438} with $r = 1/4$ ,
and conclude the left-hand side inequality in (\ref{eq_1725}). For the ``Moreover''
part, we set
$$ T = c_0 \kappa_n^{-2} / \log n, $$
for an appropriate universal constant $c_0 > 0$ to be determined shortly.
Recalling that $\kappa_n \geq 1$, we apply Proposition \ref{lem161} and conclude that
\begin{equation}  \PP \left( \forall 0\leq t\leq T, \quad \frac{1}{2}\id \leq A_t \leq 2\id  \right) \geq 1 - C\exp(-c/T) \geq 1 - \frac{1}{ n^{20}}, \label{eq_2017} \end{equation}
for a suitable choice of the universal constant $c_0 > 0$. Denote by $\cF$ the event whose probability is estimated in (\ref{eq_2017}).
Then under the event $\cF$, for $0 \leq t \leq T$,
$$ v_t = \int_0^t \Sigma_s d B_s = \int_0^t A_s d B_s = a_t, $$
completing the proof. \end{proof}

The Kahane inequalities are
reverse H\"older inequalities for norms of (say) a Gaussian random vector; see, e.g., Milman and Schechtman \cite[Appendix III]{MS}.
We were not able to find in the literature a concise version for {\it gauge functions},
which as in Rockafellar \cite[Section I.4]{roc} are convex functions $F: \RR^n \rightarrow [0, \infty)$
satisfying $F(\lambda x) = \lambda F(x)$ for all $x \in \RR^n, \lambda \geq 0$.
See Fradelizi \cite{fradelizi} for closely-related versions. For completeness we  provide here a proof relying on the
Gaussian isoperimetric inequality, for which we refer the reader e.g. to Ledoux \cite[chapter 2]{ledoux} and references therein. For $t \in \RR$ we denote
$$ \Phi(t) = \int_{-\infty}^t \frac{e^{-x^2/2}}{\sqrt{2 \pi}} dx $$
and we will use the standard inequalities
\begin{equation}   \frac{e^{-t^2/2}}{ \sqrt{2 \pi} \cdot (t+1)} \leq 1-\Phi(t)  \leq \frac{e^{-t^2/2}}{\sqrt{2 \pi} \cdot t}
\qquad \qquad \qquad (t > 0). \label{eq_1327} \end{equation}

\begin{lemma}[``Kahane's inequality for gauge functions'']
Let $\| \cdot \|$ be a gauge function in $\RR^n$, and let $G$ be a standard Gaussian random vector in $\RR^n$.
Then for any $p \geq 2$,
$$ \left( \EE \| G \|^p \right)^{1/p} \leq C \sqrt{p} \cdot \EE \| G \|, $$
where $C \leq 30$ is a universal constant.
\label{lem_kahane}
\end{lemma}

\begin{proof} Denote $E = \EE \| G \|$, and consider the convex set
$$ K = \left \{ x \in \RR^n \, ; \, \| x \| \leq 10 E \right \}. $$
By the Markov-Chebyshev inequality, $ \PP( G \not \in K) \leq \EE \| G \| / (10 E) = 1/10$. Hence
\begin{equation}  \gamma(K) \geq \frac{9}{10}, \label{eq_436} \end{equation}
where $\gamma$ is the standard Gaussian probability measure in $\RR^n$.
It follows from (\ref{eq_1327}) that
\begin{equation} 1 - \Phi(1) > \frac{1}{10}. \label{eq_435} \end{equation}
Inequality (\ref{eq_435}) implies that for any half-space of the form $H_{\theta} = \left \{ x \in \RR^n \, ; \, x \cdot \theta \leq 1 \right \}$
for  $\theta \in S^{n-1}$, we have
\begin{equation}  \gamma(H_{\theta}) < \frac{9}{10}. \label{eq_1151_} \end{equation}
Therefore the half-space $H_{\theta}$ cannot contain the convex set $K$. This implies that
\begin{equation} B^n \subseteq K. \label{eq_1148} \end{equation}
Indeed, if (\ref{eq_1148}) were not true, we could pick a point $x \in B^n \setminus K$
and separate it from the convex set $K$ by a hyperplane. The Gaussian measure of the half-space containing
$K$ must be at least $9/10$, by (\ref{eq_436}). Thus, by (\ref{eq_1151_}), this half-space must contain $B^n$,
in contradiction to the hyperplane separation between the convex set $K$ and the point $x \in B^n$. Hence (\ref{eq_1148}) is proven.
By (\ref{eq_1148}) and the Gaussian isoperimetric inequality, for any $t > 1$,
\begin{equation}
 \gamma(t K) \geq \gamma \left( K + (t-1) B^n \right) \geq \Phi( \Phi^{-1}(\gamma(K)) + (t-1) )
\geq \Phi( t), \label{eq_1338} \end{equation}
as $\gamma(K) \geq 9/10$ and hence $\Phi^{-1}(\gamma(K)) \geq 1$ by  (\ref{eq_435}). Consequently, by (\ref{eq_1327}), (\ref{eq_1338})
and elementary manipulations,
\begin{align*}  \EE \| G \|^p & \leq E^p + \int_{E^p}^{\infty} \PP \left( \| G \|^p \geq t \right) dt
\leq (10 E)^p \left( 1 + p \int_{1}^{\infty} \left[ 1 - \gamma \left( s K \right) \right] s^{p-1} ds \right)
\\ & \leq (10 E)^p \left( 1 + p \int_{1}^{\infty} s^{p-1} \left[ 1 - \Phi(s) \right ] ds \right)
\leq (10 E)^p \left( 1 + p \int_0^{\infty} s^{p-2} \frac{e^{-s^2/2}}{\sqrt{2 \pi}} ds \right).
\end{align*}
The expression in the brackets equals $1 + p 2^{(p-3)/2} \Gamma( (p-1) / 2) \leq (9p)^{p/2}$, and the lemma is proven.
\end{proof}

Theorem \ref{thm1} follows from
the bound (\ref{eq_1737}) and the following theorem:

\begin{theorem} Let $X$ be an isotropic, log-concave random vector in $\RR^n$ and
let $G$ be a standard Gaussian random vector in $\RR^n$. Then for any gauge function $\| \cdot \|$ on $\RR^n$, denoting $P_n = \kappa_n \sqrt{\log n}$,
\begin{equation}  \frac{c}{P_n } \cdot \EE \| G \| \leq \EE \| X \| \leq C P_n \cdot \EE \| G \|,
\label{eq_2026_} \end{equation}
where $c, C > 0$ are universal constants.  \label{thm11}
\end{theorem}

\begin{proof} By a standard approximation argument, we may assume that the law of $X$
is compactly-supported. Set $ T = c_0 \kappa_n^{-2} / \log n$, where $c_0 > 0$ is the constant from
Corollary \ref{lem_2205}. Let $(B_t), (a_t)$ and $(v_t)$ be as above.
Applying Corollary \ref{lem_2205} with the convex function  $F(x) = \| x \|$, we see that
\begin{equation}  \EE \| B_t/2 \| \leq \EE \| v_t \|. \label{eq_2209}
\end{equation}
The function $H(x) = \| x \|^2$ is convex as well.
By applying Corollary \ref{lem_2205} with this convex function
we obtain
\begin{equation}  \EE \| v_t \|^2 \leq \EE \| 2 B_t \|^2 = 4 t \EE \| G \|^2 \leq 720 t (\EE \| G \|)^2
= 720 (\EE \| B_t \|)^2 \label{eq_1428} \end{equation}
where we used Kahane's inequality in the form of Lemma \ref{lem_kahane}.
Write $\cF$ for the event that $a_t = v_t$ for all $0 \leq t \leq T$. By (\ref{eq_1428}) and
the ``Moreover'' part of Corollary \ref{lem_2205}, for any $0 \leq t \leq T$,
\begin{align*} \EE \| v_t \| & \leq \EE \| v_t \| 1_{\cF} + \EE \| v_t \| 1_{\cF^c} \leq \EE \| a_t \| 1_{\cF} +
\sqrt{\EE \| v_t \|^2 \cdot \frac{1}{n^{20}}} \leq \EE \| a_t \| + \frac{1}{3} \EE \| B_t \|.
\end{align*}
We combine this with (\ref{eq_2209}) and obtain
\begin{equation}  \frac{1}{6} \EE \| B_t \| \leq \EE \| a_t \|. \label{eq_1432} \end{equation}
Since $(a_t)_{t \geq 0}$ is a bounded martingale and the gauge function $\| \cdot \|$ is convex,
\begin{equation}  \EE \| a_t \| \leq \EE \| a_{\infty} \| = \EE \| X \|, \label{eq_1437}
 \end{equation}
 where we used Lemma \ref{lem190} in the last passage. By setting $t = T$ and using (\ref{eq_1432})
 and (\ref{eq_1437}), we deduce the left-hand side inequality in (\ref{eq_2026_}).
 The right-hand side inequality in (\ref{eq_2026_}) follows immediately from Theorem 1 in Eldan and Lehec \cite{eldan_lehec}.
 Indeed, for $x \in \RR^n$ set $$ N(x) = \| x \| + \| -x \|. $$ Then $N: \RR^n \rightarrow [0, \infty)$ is a non-negative, even, convex, $1$-homogeneous function on $\RR^n$, i.e., a seminorm. The proof of Theorem 1 in \cite{eldan_lehec} applies for seminorms and not just for norms, and hence
 $$ \EE \| X \| \leq \EE N(X) \leq C P_n \EE N(G) = 2 C P_n \EE \| G \|. $$
 This completes the proof of (\ref{eq_2026_}).
\end{proof}

\begin{remark} \label{rem_1814} The convexity of the gauge function $\| \cdot \|$ is crucial for our argument,
as is the fact that the gauge function satisfies reverse H\"older inequalities. The conclusion
of Theorem \ref{thm11} applies to any non-negative {\it convex} function satisfying appropriate reverse H\"older inequalities with
respect to centered, log-concave random vectors; for instance, to norms raised to a fixed, positive power, or to convex non-negative polynomials of a fixed degree.

\medskip  The conclusion of Theorem \ref{thm11} cannot possibly apply for all convex functions. Indeed,
there are convex functions that grow like $\exp(|x|^{3/2})$
at infinity, so they are integrable with respect to a Gaussian but not with respect to all isotropic, log-concave distributions.
See van Handel \cite{vh} for a convex-order comparison principle in the sub-Gaussian
case, building upon the Fernique-Talagrand theory and the recent work by Liu \cite{Liu}.
\end{remark}

\begin{remark} \label{rem_1102}
If the KLS conjecture holds true, then each of the inequalities
in (\ref{eq_2026}) is tight, up to the value of the universal constant $c > 0$.
Indeed, if $X$ is uniformly distributed in the cube $[-\sqrt{3}, \sqrt{3}]^n$
while $Y$ is a random vector with density $2^{-n/2} \exp(-\sqrt{2} \sum_{i=1}^n |x_i|)$ in $\RR^n$,
then $X$ and $Y$ are isotropic and log-concave with
 $$ \EE \| X \|_{\infty} \sim 1, \qquad \EE \| G \|_{\infty} \sim \sqrt{\log n}, \qquad
 \EE \| Y \|_{\infty} \sim \log n, $$
 Thus, in the context of Theorem \ref{thm1}, there are examples of isotropic, log-concave random vectors $X$ and $Y$ in $\RR^n$ where $\EE \| Y \| / \EE \| G \| > c \sqrt{\log n}$
while $\EE \| G \| / \EE \| X \| > c \sqrt{\log n}$.
\end{remark}

\begin{remark} The log-concavity assumption plays an essential role in Theorem \ref{thm11}. For instance, by
Klartag and Koldobsky \cite{KK}, there exists an isotropic, subexponential
random vector $X$ supported on a centrally-symmetric convex set $K \subseteq \RR^n$
of volume $L^{-n}$ for $L \sim n^{1/4}$, and one may show that
$$ \frac{\EE \| X \|_K}{\EE \| G \|_K} \sim n^{-1/4}, \qquad \frac{\EE \| X \|_{K^{\circ}}}{\EE \| G \|_{K^{\circ}}} \sim n^{1/4}. $$
I.e., a polynomial loss occurs rather than only a polylogarithmic one.
\end{remark}

\section{Using the $M M^*$ estimate}
\label{sec4}

\subsection{Proofs of the main results}

Let $K \subseteq \RR^n$ be a convex body containing the origin in its interior,
and let $\| x \|_K = \inf \{ \lambda > 0 \, ; \, x \in \lambda K \}$ be the associated
Minkowski functional. Let $G$ be a standard Gaussian random vector in $\RR^n$.
By integrating in polar coordinates, we see that
\begin{equation} \EE \| G \|_K = \alpha_n \int_{S^{n-1}} \| x \|_K d \sigma_{n-1}(x) = \alpha_n M(K) \label{eq_1616} \end{equation}
where $\alpha_n = \EE |G| \sim \sqrt{n}$. Recall that inequality (\ref{eq_1115}) of Corollary \ref{cor_1617} is proven in E. Milman \cite{e_milman}. Since
 $\kappa_n \leq 2 \psi_n$ by (\ref{eq_1737}), the remaining inequality of Corollary \ref{cor_1617}
follows from the following:

\begin{lemma} Let $K \subseteq \RR^n$ be a convex body in isotropic position. Then,
$$ M(K) \leq C \frac{\kappa_n \sqrt{\log n}}{\sqrt{n}}, $$
where $C > 0$ is a universal constant. \label{lem_2052}
\end{lemma}

\begin{proof} Let $X$ be distributed uniformly in the convex body $K$.
From (\ref{eq_1616}) and Theorem \ref{thm11},
$$ c \sqrt{n} M(K) \leq \EE \| G \|_K \leq C \kappa_n \sqrt{\log n} \EE\| X \|_K \leq C \kappa_n \sqrt{\log n} $$
since almost surely $\| X \|_K \leq 1$ as $X \in K$.
\end{proof}

Corollary \ref{cor_1617} is thus proven.  For two convex bodies $K, T \subseteq \RR^n$ containing the origin in their interior
and for a linear map $A: \RR^n \rightarrow \RR^n$ we write
$$ \| A: K \to T \| = \sup_{0 \neq x \in \RR^n} \frac{ \| A x \|_T }{\| x \|_K}. $$
Observe that $\| A: K \to T \|$ is a convex function of the linear map $A$, since it is the supremum of convex functions
of $A$. For $K \subseteq \RR^n$ denote $$ R(K) = \max_{x \in K} |x|. $$

\begin{lemma}[Chevet \cite{chevet}] Let $K, T \subseteq \RR^n$ be convex bodies containing the origin in their interior.
Let $\Gamma \in \RR^{n \times n}$ be a random matrix whose entries are independent, standard Gaussian random variables.
Then,
\begin{equation}  \EE \| \Gamma: K \to T \| \leq C \sqrt{n} \left[ R(K) M(T) + R(T^{\circ}) M(K^{\circ}) \right], \label{eq_1745} \end{equation}
where $C > 0$ is a universal constant.
\label{lem_1743}
\end{lemma}

\begin{proof} For completeness, we verify that the proof in the centrally-symmetric case applies in the general
case. The proof in Chevet \cite{chevet} (see also Benyamini and Gordon \cite[Lemma 1.1]{BG} and Gordon \cite{Gor} for better constants) in the centrally-symmetric case is based on comparison
between two centered Gaussian processes indexed by $K \times T^{\circ}$. The first Gaussian process is
$$ P_{x,y} = \langle \Gamma x, y \rangle \qquad \qquad \qquad (x \in K, y \in T^{\circ}). $$
Clearly $$  \sup_{(x,y) \in K \times T^{\circ}} P_{x,y} = \| \Gamma: K \rightarrow T \|. $$ The second Gaussian process is
$$ Q_{x,y} = \left[ R(T^{\circ}) \langle \Gamma^{(1)}, x \rangle + R(K) \langle \Gamma^{(2)}, y \rangle \right], $$
where $\Gamma^{(1)}, \Gamma^{(2)}$ are two independent, standard Gaussian random vectors in $\RR^n$.
Note that
$$ \sup_{(x,y) \in K \times T^{\circ}} Q_{x,y} =
\left[ R(T^{\circ})  \| \Gamma^{(1)} \|_{K^{\circ}} + R(K) \| \Gamma^{(2)} \|_T \right]. $$
Hence, by (\ref{eq_1616}),
$$ \EE \sup_{(x,y) \in K \times T^{\circ}} Q_{x,y} \leq C \sqrt{n}
\left[ R(T^{\circ})  M( K^{\circ} ) + R(K) M(T) \right]. $$
Observe that for $(x_1,y_1), (x_2, y_2) \in K \times T^{\circ}$,
\begin{align} \label{eq_445} \EE |P_{x_1,y_1} - P_{x_2,y_2}|^2 & = \left| x_1 \otimes y_1 - x_2 \otimes y_2 \right|^2 \leq
2 \left[ |x_1|^2 |y_1 - y_2|^2 + |x_1 - x_2|^2 |y_2|^2 \right] \\ & \leq 2 \left[ R^2(T^{\circ}) |x_1 - x_2|^2 + R^2(K) |y_1- y_2|^2
\right] = 2 \EE \left| Q_{x_1,y_1} - Q_{x_2,y_2} \right|^2. \nonumber
\end{align}
By using comparison
theorems for suprema of Gaussian processes (e.g. Ledoux and Talagrand \cite[Corollary 3.14]{LT}), we deduce from (\ref{eq_445}) that
$$ \EE \sup_{(x,y) \in K \times T^{\circ}} P_{x,y} \leq C \EE \sup_{(x,y) \in K \times T^{\circ}} Q_{x,y}, $$
and inequality (\ref{eq_1745}) follows.
\end{proof}

The next step is to replace the Gaussian matrix $\Gamma$ from  Lemma \ref{lem_1743}
by a random orthogonal matrix $U$.
This replacement is standard, and it is described in Benyamini and Gordon \cite{BG}, in Bourgain and Milman \cite{BouM} following Marcus and Pisier \cite{MaPi},
and in Davis, Milman and Tomczak-Jaegermann \cite{DMTJ}.
For completeness we  provide a proof of the following:

\begin{corollary} Let $K, T \subseteq \RR^n$ be convex bodies containing the origin in their interior.
Let $U$ be a random orthogonal matrix, distributed uniformly in the orthogonal group $O(n)$.
Then,
$$ \EE \| U: K \to T \| \leq C \left[ R(K) M(T) + R(T^{\circ}) M(K^{\circ}) \right], $$
where $C > 0$ is a universal constant. \label{cor_2033}
\end{corollary}

\begin{proof} We follow Benyamini and Gordon \cite{BG} who treat the centrally-symmetric case.
Let $\Gamma \in \RR^{n \times n}$ be a random matrix whose entries are independent, standard Gaussian random variables,
that are also independent of $U$. Observe that
$$ (\Gamma^* \Gamma)^{1/2} U \qquad \text{and} \qquad \Gamma $$
are equidistributed. As is computed in Benyamini and Gordon \cite[Lemma 1.12]{BG}, we have
$$ \EE ( \Gamma^* \Gamma)^{1/2} = \delta_n \id $$
with $\delta_n \geq c \sqrt{n}$. Since $\| A: K \to T \|$ is a convex function
of the matrix $A$, by Jensen's inequality,
\begin{align*}  \EE \| U: K \to T \| & = \delta_n^{-1} \EE_U \| \EE_\Gamma ( \Gamma^* \Gamma)^{1/2} U  : K \to T \|
\\ & \leq \delta_n^{-1} \EE_{U, \Gamma} \| ( \Gamma^* \Gamma)^{1/2} U  : K \to T \| = \delta_n^{-1} \EE_\Gamma \| \Gamma: K \to T \|. \end{align*}
The result now follows from Lemma \ref{lem_1743} and from the fact that $\delta_n \geq c \sqrt{n}$.
\end{proof}

\begin{remark} \label{rem_1045} We record here a variation of the above argument which will be useful for applications to linear symplectic geometry later on. We identify $\CC^n \cong \RR^{2n}$ and write $J(z) = i z$ for $z \in \CC^n$.
Thus $J: \RR^{2n} \rightarrow \RR^{2n}$ is a linear map. It was proven by Gluskin and Ostrover \cite{GO_rotation} that for any centrally-symmetric convex body $K \subseteq \RR^{2n}$,
\begin{equation}  \EE \| J : (U K)^{\circ} \rightarrow U K \| = \EE \| U^* J U: K^{\circ} \rightarrow K \| \leq
 C R(K^\circ)M(K), \label{eq_1036} \end{equation}
where $U$ is a random orthogonal matrix, distributed uniformly in $O(2n)$,  and $C>0$ is a universal constant.
Indeed, by Corollary 3.3 in \cite{GO_rotation},
$$ \EE \| U^* J U: K^{\circ} \rightarrow K \| \lesssim \frac{1}{\sqrt{n}}\EE \left \| \Gamma: K^{\circ} \to K \right \|$$
where, as previously, $\Gamma \in \RR^{2n \times 2n}$ is a random matrix whose entries are independent, standard Gaussian random variables. We may now apply Lemma \ref{lem_1743} to prove \eqref{eq_1036}.
\end{remark}

\begin{proof}[Proof of Theorem \ref{thm00}] We will prove the theorem with $\alpha = 2$.
By the Markov-Chebyshev inequality, it suffices to show that
$$ \EE \|U: K_1 \rightarrow K_2 \| \leq C \sqrt{n} \log^{2} n. $$
This would follow from Corollary \ref{cor_2033}, once we show that
\begin{equation}
R(K_1) M(K_2) + R(K_2^{\circ}) M(K_1^{\circ}) \leq C \sqrt{n} \log^{2} n.
\label{eq_2049} \end{equation}
However, $R(K_1) \leq C n$ and $R(K_2^\circ) \leq C$ by (\ref{eq_1935}).
From Corollary \ref{cor_1617} and (\ref{eq_1100}) we know that
$$ M(K_2) \leq C \frac{\psi_n \sqrt{\log n}}{\sqrt{n}} \leq \tilde{C} \frac{\log n}{\sqrt{n}} $$
and
$$ M(K_1^{\circ}) = M^*(K_1) \leq C \log^2 n \cdot \sqrt{n}. $$
This proves (\ref{eq_2049}).
\end{proof}

\begin{proof}[Proof of Theorem \ref{thm0}] We may apply an affine transformation,
and assume that both $K_1$ and $K_2$ are in isotropic position.
Let $U \in O(n)$ be a random orthogonal transformation, distributed uniformly in $O(n)$.
Observe that $U^{-1}$ is again distributed uniformly in $O(n)$. Hence, by Theorem~\ref{thm00},
for
$$ \lambda = C \sqrt{n} \log^{2} n, $$
with probability of at least $4/5$,
$$ U(K_1) \subseteq \lambda K_2 \qquad \text{and} \qquad U^{-1}(K_2) \subseteq \lambda K_1. $$
We have thus found $U \in O(n)$ such that
$$ U(K_1) \subseteq \lambda K_2 \subseteq \lambda^2 U(K_1), $$
and the theorem follows with $\alpha = 4$.
\end{proof}

\begin{proof}[Proof of Theorem \ref{thm_2118}] We will prove the theorem with $\alpha = 2$.
Let $X$ be a random vector distributed uniformly in $T$,
and let $Y$ be a random vector distributed uniformly in $K$.
Then both $X$ and $Y$ are log-concave and isotropic. Let $G$ be a standard Gaussian random vector in $\RR^n$.
Since $\psi_n \leq C \sqrt{\log n}$ by (\ref{eq_1100}),
we may apply Theorem \ref{thm1} twice and conclude that
$$ \EE \| X \|_K \leq C \log n \EE \| G \|_K \leq \tilde{C} \log^2 n \EE \| Y \|_K \leq \tilde{C} \log^2 n, $$
where in the last passage we used that $\| Y \|_K \leq 1$ almost surely. By the Markov-Chebyshev inequality,
with $\lambda = 100 \tilde{C} \log^2 n$,
$$ \frac{ \Vol_n(\lambda K \cap T) }{\Vol_n(T)} = \PP \left( \| X \|_K \leq 100 \tilde{C} \log^2 n \right) \geq \frac{99}{100}, $$
which completes the proof.

\medskip Let us now explain how to replace the value $0.99$ by $(1-\varepsilon)$ for some $\varepsilon \in (0,1)$, at the expense of replacing
$\lambda$ by an appropriate $\lambda_{\eps}$.

\medskip Consider first the case where $K$ is centrally-symmetric. In this case, by Borell's lemma, for $r\geq 1$,
$$ \PP \left( \| X \|_K \geq r\EE \| X \|_K\right) \leq 2\exp\left(-cr\right). $$
For Borell's lemma, see e.g., Milman and Schechtman \cite[Appendix III]{MS}.
In this case, we may thus choose
\begin{equation} \lambda_{\varepsilon} = C'\log\left(\frac{2}{\varepsilon}\right) \log^2 n. \label{eq_2128}
\end{equation}

\medskip Next, consider the case where $T$ is centrally-symmetric. In this case, denote  $N(x) = \| x \|_K + \| -x \|_K$
and observe that $\EE N(x) = 2 \EE \| X \|_K$. By Borell's lemma,
$$ \PP \left( \| X \|_K \geq 2r\EE \| X \|_K\right) \leq \PP \left( N(x) \geq r\EE N(x) \right) \leq 2\exp\left(-cr\right) $$
and we may thus choose $\lambda_{\varepsilon}$ as in (\ref{eq_2128}).

\medskip In the case where neither $T$ nor $K$ is centrally-symmetric, we need another argument. We argue that in the general case, for $r \geq 1$,
\begin{equation} \PP \left( \| X \|_K \geq r\sqrt{\log n}\,\EE \| X \|_K\right) \leq C\exp\left(-cr\right).
\label{eq_1853} \end{equation}
It follows from (\ref{eq_1853}) that we may choose
\[
\lambda_\varepsilon
=
C''\log\!\left(\frac{2}{\varepsilon}\right)
(\log n)^{5/2}.
\]
In the remainder of this proof let us prove (\ref{eq_1853}),
along the lines of the proof of
Lemma \ref{lem_kahane}.
First, we claim that if $L \subseteq \RR^n$ is a convex subset satisfying
\begin{equation}
\mathbb{P}(X \in L)\geq \frac34,
\label{eq_1908} \end{equation}
then
\begin{equation}
cB^n \subseteq L. \label{eq_1906}
\end{equation}
Indeed, if (\ref{eq_1906}) does not hold true, then there is a direction $\theta \in S^{n-1}$ such that $P_\theta L\subset (-\infty,c]$, with $P_{\theta}(x) = x \cdot \theta$.
This implies that
\[
\mathbb{P}(X_\theta\in(-\infty,c])\geq \frac34,
\]
where $X_{\theta} = X \cdot \theta$.
On the other hand,
\[
\begin{aligned}
\mathbb{P}(X_\theta\in(-\infty,c])
&=\mathbb{P}(X_\theta\in(-\infty,0])
  +\mathbb{P}(X_\theta\in[0,c])\\
&\leq 1-\frac1e+c\|f_\theta\|_\infty\\
&\leq 1-\frac1e+c,
\end{aligned}
\]
where $f_\theta$ denotes the density of $X_\theta$. In the second line we used Gr\"unbaum's inequality, using the centering of $X$, while the second line is a standard consequence of the fact that $\var(X_\theta)=1$, see e.g \cite[Section 5]{LV} for both results. The claim is proved with $c<\frac1e-\frac14$.

\medskip Let us now apply the above claim with
$L=4\,\mathbb{E}\|X\|_K\,K$, which satisfies (\ref{eq_1908}) by the Markov-Chebyshev inequality.
We conclude that
\[
cB^n\subseteq L.
\]
By Gromov and Milman \cite{GM}, exponential concentration follows from the Poincar\'e inequality, thus,
\[
\mathbb{P}(X\notin L+r C B^n)
   \leq 2\exp\!\left(
      -\frac{rc}{\sqrt{C_P(X)}}
   \right).
\]
Since $X$ is isotropic, $\sqrt{C_P(X)}
\leq \psi_n
\leq C\sqrt{\log n}$
by (\ref{eq_1100}) above. It follows that for any $r\geq 1$,
\[
\mathbb{P}\!\left(
\|X\|_K
\geq
C r\sqrt{\log n}\,
\mathbb{E}\|X\|_K
\right)
\leq
2e^{-r},
\]
proving (\ref{eq_1853}).
\end{proof}

\begin{remark}
One may slightly improve the logarithmic exponent in Theorem~\ref{thm0} by invoking Rudelson's estimate relating the Banach--Mazur distance to the $MM^*$-product \cite[Theorem~5]{rudelson}. More precisely, if $K_1,K_2$ are in isotropic position and if
\[
M=\max\{M(K_1),M(K_2)\},
\qquad
M^*=\max\{M^*(K_1),M^*(K_2)\},
\]
then the argument above gives
\[
d_{BM}(K_1,K_2)\leq C (nM+M^*)^2,
\]
and the distance is realized in the random isotropic position. On the other hand, Rudelson's theorem gives
\begin{equation}
d_{BM}(K_1,K_2)\leq C n\sqrt{\log n}\,MM^* .
\label{eq_1953} \end{equation}
Combined with Corollary~\ref{cor_1617} and Milman's estimate~\eqref{eq_1100}, this would improve the exponent in Theorem~\ref{thm0} from $\alpha=4$ to $\alpha=\frac{7}{2}$.

\medskip The drawback of this improvement is that it is not obtained in the isotropic position, but rather in a somewhat ad-hoc position that mixes the John position
and a position with controlled $M M^*$. The improvement comes from the imbalance between the currently available bounds on $nM(K)$ and $M^*(K)$ in isotropic position. If the two estimates were of comparable size,
i.e. if $nM \sim M^*$, then the bound (\ref{eq_1953}) would be worse by a factor of $\sqrt{\log n}$ when compared with the random isotropic position above.
\label{rmk_exponent}
\end{remark}

\begin{proof}[Proof of Corollary \ref{cor_2109}] We prove the corollary with $\alpha = 4$.
We may assume that $K_1$ and $K_2$ are in isotropic position. Theorem \ref{thm_2118}
now implies the desired result.
\end{proof}

We now provide a lower bound on the partial containment distance between the cube and the Euclidean ball in high dimensions.
\begin{lemma}\label{lem_1079}
Let
$K_1=B_{\infty}^n = [-1,1]^n$ and $K_2=B^n$. Then,
\[
 d_{PC}(K_1,K_2)\ge c\sqrt{\log n},
\]
for a universal constant $c>0.$
\end{lemma}

\begin{proof}
For centrally symmetric convex bodies $K,L$,
the function $z\mapsto \Vol_n( K\cap (L+z) )$ is even and log-concave, hence it is maximized
at $z=0$. Thus when both bodies are symmetric the translations involved in the definition of $d_{PC}$ may be discarded. Let $A$ be an invertible matrix such that, writing $P=AB_\infty^n$,
\begin{equation}
 \Vol_n( P\cap sB^n ) \ge 0.99 \Vol_n(P),
 \qquad
 \Vol_n( sP\cap B^n ) \ge 0.99 \Vol_n( B^n).
\label{eq1096}
\end{equation}
Our goal is to prove that $s^2\geq c\sqrt{\log n}$.

\medskip 
We set $H=\|A\|_{\mathrm{HS}}$ and we claim that $H\lesssim s$. Indeed, if $U$ is uniformly distributed on
$B_\infty^n$ and $Z=|AU|^2$, then by Khinchine's inequality,
\[
 \EE Z=\frac{H^2}{3},
 \qquad
 \EE Z^2\le 3(\mathbb EZ)^2.
\]
By the Paley--Zygmund inequality,
\[
 \PP\left(|AU|\geq H/\sqrt6\right)\geq \frac{1}{12}.
\]
On the other hand, the first inequality in \eqref{eq1096} gives
$\PP(|AU|\leq s)\geq 0.99$. Thus, necessarily,
\begin{equation}
 H\leq \sqrt{6} s .
 \label{eq1115}
\end{equation}

We now use the second inequality in \eqref{eq1096}, which requires a bit more work.
Let $X$ be a random vector that is distributed uniformly on $B^n$. The second inequality in \eqref{eq1096} says that
\[
 \PP\left(\|X\|_P\leq s\right)\geq 0.99.
\]
Writing $X=R\Theta$, where $\Theta$ is uniform on $S^{n-1}$ and $R$ is
independent of $\Theta$, and using $\PP(R\geq 1/2)=1-2^{-n}$, we get
\[
 \PP_\Theta\left(\|\Theta\|_P\leq 2s\right)\geq 0.98
\]
for all $n$ large enough. Let $G$ be a standard Gaussian random vector in $\RR^n$.
Since $G/|G|$ is uniform on $S^{n-1}$ and
$\PP(|G|\le 3\sqrt n)\geq 8/9$, it follows that
\[
 \PP\left(\|n^{-1/2}G\|_P\leq 6s\right)\geq 3/4.
\]
Arguing as in the proof of Theorem \ref{thm_2118}, it is easy to check that for Gaussian norms, a $3/4$-quantile controls the mean up to a universal
constant. Thus,
\begin{equation}
 \EE\|n^{-1/2}G\|_P
 =
 \EE\|n^{-1/2}A^{-1}G\|_\infty
 \leq Cs .
 \label{eq1140}
\end{equation}

It remains to lower-bound the $\ell_\infty$ norm of a general Gaussian vector. Put $B=A^{-1}$, let
$b_1,\ldots,b_n$ be the rows of $B$, and let $a_1,\ldots,a_n$ be the columns of
$A$. Since $b_i\cdot a_j=\delta_{ij}$,
$$
 \operatorname{dist}\left(b_i,\operatorname{span}\{b_j:j\neq i\}\right)
 =
 \frac{1}{|a_i|}.
$$
Using that $\sum_i|a_i|^2=H^2$, there is a set $I\subset\{1,\ldots,n\}$ with
$|I|\geq n/2$ such that, for every $i\in I$,
$$
 \operatorname{dist}\left(b_i,\operatorname{span}\{b_j:j\neq i\}\right)
 \geq
 \frac{\sqrt n}{\sqrt2 H}
 =:\sigma .
$$
In particular, $|b_i-b_j|\ge \sigma$ for all distinct $i,j\in I$. By
the Sudakov minoration principle \cite[Theorem 3.18]{LT}, applied to the Gaussian process
$Y_i=\langle b_i,G\rangle$, $i\in I$,
\[
 \EE\max_{i\in I}\langle b_i,G\rangle
 \geq c\sigma\sqrt{\log |I|}.
\]
Therefore
\begin{equation}
 \EE\|n^{-1/2}A^{-1}G\|_\infty
 \geq
 c\frac{\sqrt{\log n}}{H}.
\label{eq1171}
\end{equation}
Combining \eqref{eq1140} and \eqref{eq1171} gives
\[
 sH\geq c\sqrt{\log n}.
\]
Together with \eqref{eq1115}, this yields
\[
 s^2\geq c\sqrt{\log n}.
\]
which is what we wanted to prove.
\end{proof}

\subsection{Some applications}
We proceed with an application to linear symplectic geometry,
where the Ekeland-Hofer-Zehnder capacity (EHZ) plays a role.
As in Remark \ref{rem_1045}, we identify $\RR^{2n} \cong \CC^n$ and use the linear operator $J: \RR^{2 n} \rightarrow \RR^{2n}$ that corresponds
to multiplication by $i$. It is proven in Gluskin and Ostrover \cite{GO} that for a centrally-symmetric
convex body $K \subseteq \RR^{2n}$, the expression
$$ \frac{1}{\| J: K^{\circ} \rightarrow K \|} $$
differs from the EHZ capacity $c_{EHZ}(K)$ by a factor of at most $4$. Proposition 3.2 in \cite{GO} implies that
for the unit cube $K = [-1,1]^{2n}$ there is a rotation $U \in O(2n)$ with
$$ c_{EHZ}(U K) \geq c \sqrt{n}. $$
Up to logarithmic factors, this is a general phenomenon of the random isotropic position:

\begin{proposition} Let $K \subseteq \RR^{2n}$ be a centrally-symmetric convex body in isotropic position,
and let $U$ be a random rotation, distributed uniformly in $O(2n)$. Then,
$$ \EE c_{EHZ}(U K) \geq c \frac{\sqrt{n}}{\log n}, $$
where $c > 0$ is a universal constant.
\end{proposition}

\begin{proof} By Jensen's inequality and Remark \ref{rem_1045},
$$ \EE c_{EHZ}(U K) \geq \frac{1}{4} \EE \frac{1}{\| J: (U K)^{\circ} \rightarrow UK \|} \geq
\frac{1}{4 \EE \| J: (U K)^{\circ} \rightarrow UK \| } \geq \frac{c}{ R(K^\circ)M(K) }. $$
By (\ref{eq_1935}) we know that $R(K^{\circ}) \leq 1$ while $M(K) \leq C \log n / \sqrt{n}$ by Corollary
\ref{cor_1617} and (\ref{eq_1100}).
\end{proof}

\medskip We conclude with an application to the first Dirichlet eigenvalue of the Laplacian.
For a bounded domain $K\subseteq\RR^n$, let $\lambda_1(K)$ denote the
minimal eigenvalue of $-\Delta$ on $K$ with Dirichlet boundary conditions, i.e.,
$$ \lambda_1(K) = \inf_{0 \not \equiv u \in C_c^{\infty}(K)} \frac{\int_{K} |\nabla u|^2}{\int_K u^2}, $$
where the infimum runs over all smooth functions that are compactly-supported in $K$ and are not identically zero.
The Faber-Krahn inequality states that
\begin{equation}  \lambda_1(K) \geq \lambda_1(rB^n)
 \geq  c n\,\Vol_n(K)^{-2/n}, \label{eq_1418} \end{equation}
 where $r = ( \Vol_n(K) / \Vol_n(B^n) )^{1/n}$.
 See Schmuckenschl\"ager \cite{S} for the right-hand side inequality in (\ref{eq_1418}),
 as well as for the conjecture that for any centrally-symmetric convex body $K \subseteq \RR^n$ there exists an invertible
 linear map $T: \RR^n \rightarrow \RR^n$ such that $\tilde{K} = T(K)$ satisfies
 $$ \lambda_1( \tilde{K} ) \leq  C n \,\Vol_n( \tilde{K} )^{-2/n} $$
 for a universal constant $C > 0$.
 The following is the reverse Faber-Krahn argument of Klartag and E. Milman,
recorded in Maz'ya \cite[Section 4.9]{mazya} and in Naor
\cite[Section 1.6.3 and Proposition 175]{naor}.

\begin{proposition}
Let $K\subseteq\RR^n$ be a convex body containing the origin in its interior. Then
\[
 \lambda_1(K)\le C n^2 M(K)^2.
\]
Consequently, if $K$ is in isotropic position,
then
\begin{equation}
 c n \leq \lambda_1(K)\leq C' n\psi_n^2\log n \le \tilde{C} n\log^2 n.
\label{eq_1354} \end{equation}
Here, $c, C, C', \tilde{C} > 0$ are universal constants.
\end{proposition}

\begin{proof} Denote $$ r=\frac{1}{2M(K)} $$
and set
\[
 L=K\cap rB^n .
\]
By the Chebyshev-Markov inequality,
\[
 \sigma_{n-1}\{\theta\in S^{n-1}:\|\theta\|_K\le 2M(K) \}\ge \frac12 .
\]
Therefore, by integration in polar coordinates,
\[
 \Vol_n(L)  = \Vol_n(B^n) \int_{S^{n-1}} \| \theta \|_L^{-n} \,d\sigma_{n-1}(\theta)
 \geq  \frac12 \Vol_n(r B^n).
\]
Consequently, since surface area is monotone under inclusion in the class of convex bodies,
$$ \frac{\Vol_{n-1}(\partial L)}{\Vol_n(L)} \leq \frac{\Vol_{n-1}(\partial (r B^n))}{\Vol_n(L)} \leq
2 \frac{\Vol_{n-1}(\partial (r B^n))}{\Vol_n(r B^n)} = \frac{2 n}{r} \leq C n M(K).  $$
By P\'olya's inequality for convex domains \cite{polya} and its high-dimensional form
due to Jo\'o and Stach\'o \cite{JS},
\[
 \lambda_1(L)
 \le
 C\left(\frac{\Vol_{n-1}(\partial L)}{\Vol_n(L)}\right)^2
 \le
 C n^2M(K)^2 .
\]
Since $L\subseteq K$, the monotonicity of the Dirichlet eigenvalue gives
\[
 \lambda_1(K)\le \lambda_1(L),
\]
proving the first assertion. We move to the proof of (\ref{eq_1354}).
Assume that $K$ is in isotropic position.
By the ``easy direction'' of Bourgain's slicing problem, we have $L_K \geq c$
where $L_K = \left( \det \cov(K) / \Vol_n(K)^2 \right)^{1/(2n)}$ is the isotropic constant of $K$.
Hence $$ \Vol_n(K)^{1/n} \leq C. $$ Thus the left-hand side inequality in (\ref{eq_1354})
follows from the Faber-Krahn inequality (\ref{eq_1418}). Next, since $K$ is in isotropic position, Corollary \ref{cor_1617} yields
\[
 M(K)\le C\frac{\psi_n\sqrt{\log n}}{\sqrt n}.
\]
Therefore
\[
 \lambda_1(K)
 \le
 C n^2 M(K)^2
 \le
 C n\psi_n^2\log n \leq \tilde{C} n \log^2 n,
\]
where we used \eqref{eq_1100} in the last passage. This completes the proof of (\ref{eq_1354}).
\end{proof}

\appendix
\section{Appendix: Proof of Proposition \ref{lem161}}

We denote the maximal and minimal eigenvalues of a matrix $A \in \RR^{n \times n}_{sym}$
by $\lambda_{\max}(A)$ and $\lambda_{\min}(A)$, respectively.
Following \cite{klartag2022bourgain}, we will use  proxies for $\lambda_{\max}$ and $\lambda_{\min}$.
For a symmetric matrix $A \in \RR_{sym}^{n \times n}$ and a parameter $\beta>0$ define
$$f_\beta(A) = \frac{1}{\beta}\log\tr\exp(\beta A)$$
and
$$g_\beta(A) = -f_\beta(-A) = -\frac{1}{\beta}\log\tr\exp(-\beta A).$$
These proxies satisfy the inequalities
\begin{equation}
    \lambda_{\max}(A) \,\leq\, f_\beta(A) \,\leq\, \lambda_{\max}(A) + \frac{\log n}{\beta}
\end{equation}
and
\begin{equation}
    \lambda_{\min}(A) \,\geq\, g_\beta(A)\,\geq\, \lambda_{\min}(A)-\frac{\log n}{\beta}.
\end{equation}

\begin{lemma}\label{lem184}
    Consider the function $\Phi: \RR^{n \times n}_{sym} \rightarrow \RR$ defined via
    $$\Phi(A) = \tr \left[  e^A \right].$$
    Then its gradient is given by
    $$\nabla\Phi(A) = e^A \in \RR^{n \times n}_{sym}. $$
    As for the Hessian of $\Phi$, for any $H\in \RR^{n \times n}_{sym}$ we have
    $$\nabla^2\Phi(A)(H,H) \leq \tr \left[ e^AH^2 \right].$$
\end{lemma}
\begin{proof}
Set $F(A)=e^{A}$. By differentiating the power series of $e^{A}$ term-wise,
\[
 dF(A)(H)
=\sum_{i,j\ge 0}\frac{1}{(i+j+1)!}A^i H A^j = \sum_{k\ge1}\frac{1}{k!}\sum_{j=0}^{k-1}A^{j}HA^{k-1-j}.
\]
Classically, this can be rewritten using an integral representation
\[ dF(A)(H)
=\int_0^1 e^{(1-s)A}\,H\,e^{sA}\,ds.
\]
By using either of these two formulae, as well as the cyclic invariance of the trace,
\[
d\Phi(A)(H)
=\Tr\left[ dF(A)(H)\right]
=\Tr[ e^{A}H ],
\]
thus $\nabla \Phi(A) = e^A$.
Differentiating once more in the direction of $H$,
\[
\nabla^2\Phi(A)(H,H)
=\Tr\left[ dF(A)(H)\,H\right]
=\int_0^1 \Tr\left[ e^{(1-s)A}H\,e^{sA}H\right]\,ds
=\int_0^1 g(s)\,ds,
\]
where $$ g(s):=\Tr\!\left[ e^{(1-s)A}H\,e^{sA}H\right].$$ We diagonalize $A=Q\Lambda Q^*$
for an orthogonal matrix $Q$ and a diagonal matrix $\Lambda$ whose diagonal elements are
$\lambda_1,\ldots,\lambda_n \in \RR$. Define also $\tilde{H}=Q^* H Q = (\widetilde H_{ij})_{i,j=1,\ldots,n}$.
Then
\[
g(s)=\sum_{i,j=1}^n e^{(1-s)\lambda_i+s\lambda_j}\,|\widetilde H_{ij}|^2.
\]
We have expressed $g$ as a sum of $n^2$ convex functions of $s$, hence $g$ itself is convex on $[0,1]$. Therefore
\[
\int_0^1 g(s)\,ds \,\leq\, \frac{1}{2}\left(g(0)+g(1)\right).
\]
We conclude that
\[
(\nabla^2\Phi)(A)(H,H)
=\int_0^1 g(s)\,ds \,\leq\,\frac{1}{2}\left(g(0)+g(1)\right) = \Tr(e^{A}H^2).
\]
\end{proof}

\begin{lemma}\label{lem228}
Let $X$ be a centered, log-concave random vector in $\RR^n$. For $\theta\in\RR^n$, define
\[
H_\theta = \EE[\langle X,\theta\rangle \,X \otimes X] \in \RR^{n \times n}_{sym}
\]
and for $1\leq i\leq n$, abbreviate $H_i = H_{e_i}$. Then,
\begin{equation}
\left \|\sum_{i=1}^n H_i^2 \right \|_{op} \leq \kappa_n^2 \| \cov(X) \|_{op}^{3},
\label{eq_2001} \end{equation}
while  for any $\theta\in S^{n-1}$,
\[
\|H_\theta\|_{op}^2 \leq 9 \|\cov(X)\|_{op}^3.
\]
\end{lemma}
\begin{proof}
For the first inequality, observe that for any vector $\theta\in S^{n-1}$,
$$ \sum_{i=1}^n \langle H_i^2 \theta, \theta \rangle = \Tr[H_{\theta}^2] = \| H_{\theta} \|_{HS}^2. $$
Now (\ref{eq_2001}) follows from (\ref{eq_1945}), as
$$ \left \|\sum_{i=1}^n H_i^2 \right \|_{op} = \sup_{\theta \in S^{n-1}} \| H_{\theta} \|_{HS}^2
\leq \| \cov(X) \|_{op}^3 \kappa_n^2. $$
The second inequality relies on comparison of moments for centered, log-concave random variables (see Eitan \cite[Theorem 4]{eitan}). Indeed, for any $\theta,u\in S^{n-1}$,
\begin{align*}
    \langle H_\theta u, u \rangle^2 &= \left(\EE\langle X,\theta\rangle\langle X,u\rangle^2\right)^2\\
    &\leq \EE\langle X,\theta\rangle^2\EE\langle X,u\rangle^4 \\
    &\leq 9 \EE\langle X,\theta\rangle^2 \left(\EE\langle X,u\rangle^2\right)^2 \leq 9 \|\cov(X)\|_{op}^3.
\end{align*}
\end{proof}

We will also need a deviation inequality for continuous martingales. The following lemma was established by Freedman \cite{freedman1975tail} in the discrete setting, and later extended by Shorack and Wellner \cite[Appendix B]{shorack2009empirical}.
\begin{lemma}\label{lem_freedman}
    Let $T > 0$ and let $(M_t)_{t\geq0}$ be a continuous martingale starting at $0$ such that
    $$[M]_T \leq b \;\; \textrm{a.s.}$$
    Then for any $a>0$,
    $$\PP\left(\sup_{t\in [0,T]} M_t \geq a\right)\leq e^{-\frac{a^2}{2b}}$$
\end{lemma}

We are now in position to prove the right-hand side inequality in (\ref{eq_1842}) of Proposition \ref{lem161}. We repeat the proof given in Klartag and Lehec \cite{KL_bulletin}.

\begin{lemma}\label{lem155}
Let $T_0 = (C\kappa_n^2\log n)^{-1}$. Then,
    $$\PP\left(\exists\, t\leq T_0,\, \|A_t\|_{op} \geq 2\right) \leq \exp\left(-\frac{c}{T_0}\right)$$
    where $c>0$ is a universal constant.
\end{lemma}

\begin{proof}
    We drop the $t$ subscript for readability. Let $\beta>0$ and recall that
    $$f_\beta(A) = \frac{1}{\beta}\log\tr\exp(\beta A) = \frac{1}{\beta} \log \Phi(\beta A).$$
    Thus
    $$\nabla f_\beta(A) = \frac{e^{\beta A}}{\Tr e^{\beta A}} = \frac{\nabla\Phi(\beta A)}{\Phi(\beta A)} =: M, $$
    where $M$ is a symmetric, positive semi-definite matrix of trace $1$. Furthermore, for any $H\in\RR^{n \times n}_{sym}$, by Lemma \ref{lem184},
    \begin{align*}
        \nabla^2f_\beta(A)(H,H) & =  \left(\frac{\beta\nabla^2\Phi(\beta A)}{\Phi(\beta A)} - \frac{\beta \nabla\Phi(\beta A) \otimes \nabla\Phi(\beta A)}{\Phi(\beta A)^2}\right)(H,H) \\
        &\leq \beta \frac{\nabla^2\Phi(\beta A)(H,H)}{\Phi(\beta A)}\\
        &\leq \beta \frac{\Tr \left[ e^{\beta A}H^2 \right]}{\Tr \left[ e^{\beta A} \right]} = \beta \, M\cdot H^2,
    \end{align*}
    where for two symmetric matrices $X, Y \in \RR^{n \times n}_{sym}$ we abbreviate $X \cdot Y := \Tr[XY]$.
    Recall that $X \cdot Y \geq 0$ when the matrices $X$ and $Y$ are positive semi-definite.
    In particular, $A \cdot M \geq 0$.    Formula \eqref{eq155} thus yields
    \begin{align}
    df_\beta(A) &= \nabla f_\beta(A)\cdot\sum_{i=1}^nH_{i}dB_i -\nabla f_\beta(A)\cdot A^2 dt + \frac{1}{2}\sum_{i=1}^n \nabla^2f_\beta(A)(H_{i}\,,\,H_{i}) \, dt\nonumber \\
    &\leq M\cdot\sum_{i=1}^nH_{i}dB_i + \frac{1}{2}\sum_{i=1}^n \nabla^2f_\beta(A)(H_{i}\,,\,H_{i}) \, dt\nonumber\\
    &\leq M\cdot\sum_{i=1}^nH_{i}dB_i + \frac{\beta}{2}M\cdot\sum_{i=1}^nH_i^2 \, dt\nonumber\\
    &\leq M\cdot\sum_{i=1}^nH_{i}dB_i + \frac{\beta}{2} \left \|\sum_{i=1}^n H_i^2 \right \|_{op}\, dt\nonumber\\
    &\leq M\cdot\sum_{i=1}^nH_{i}dB_i + \frac{\beta}{2} \| A \|_{op}^3 \kappa_n^2 \, dt\label{eq302}
    \end{align}
    where we used that $\Tr M \!=\! 1$ as well as Lemma \ref{lem228}.
    Consider the stopping time
    $$\tau = \inf \{t>0 \ ; \, \|A_t\|_{op}\geq 2\}.$$
    We choose $\beta = 2\log n$, and  set
    $$T_0 = \frac{c_0}{\kappa_n^2\log n}$$
    for some $c_0 > 0$ to be determined. Recall that $f_{\beta}(A_0) = 1 + \beta^{-1} \log n$. From \eqref{eq302}  we obtain that for any $t\leq T_0$,
    \begin{align}
        f_\beta(A_{t\wedge\tau}) &\leq N_t + 1 + \frac{\log n}{\beta} + \int_0^{t\wedge\tau} \frac{\beta}{2} \|A_s\|^3_{op}\kappa_n^2 ds \nonumber\\
        &\leq N_t + 3/2 + 8T_0\kappa_n^2\log n \nonumber\\
        &\leq 7/4 + N_t \label{eq314}
    \end{align}
    where we chose $c_0 = \frac{1}{32}$ and $N_t$ is the martingale
    $$N_t = \int_0^{t\wedge\tau}\sum_{i=1}^nM\cdot H_{i}dB_i. $$
    Its quadratic variation satisfies
    $$ [N]_t = \int_0^{t\wedge\tau} \sum_{i=1}^n \left( \Tr \left[ MH_i \right] \right)^2. $$
Since $M$ is a symmetric, positive semi-definite matrix of trace $1$, we have
    \begin{align*}
        \sum_{i=1}^n \left( \Tr \left[ MH_i \right] \right)^2 & = \sup_{\theta\in S^{n-1}}\left(\sum_{i=1}^n\Tr [MH_i\theta_i]\right)^2 \\
        & = \sup_{\theta\in S^{n-1}} \left(\Tr[MH_\theta]\right)^2\\
        & \leq  \sup_{\theta\in S^{n-1}} \|H_\theta\|_{op}^2\\
        & \leq 9\|A\|_{op}^3,
    \end{align*}
    where we used the second part of Lemma \ref{lem228}. Thus, for any $t\leq T_0$,
    $$[N]_t \leq 9\int_{0}^{t\wedge\tau} \|A_{s}\|_{op}^3 ds \leq CT_0.$$
    Recalling that $\| A_{t \wedge \tau} \|_{op} \leq f_{\beta}(A_{t \wedge \tau})$, from \eqref{eq314} we see that for any $t\leq T_0$,
    $$\|A_{t\wedge\tau}\|_{op} \leq 7/4 + N_t. $$
Thus,
\begin{align*}
    \PP\left( \tau \leq T_0 \right)
        &= \PP\left( \exists\, t \in [0,T_0] \; ; \; \|A_{t \wedge \tau}\|_{op} \geq 2 \right) \\
        &\leq \PP\left( \exists\, t \in [0,T_0] \; ; \; N_{t} \geq \tfrac{1}{4} \right) \\
        &= \PP\left( \exists\, t \in [0,T_0] \; ; \; N_{t} \geq \tfrac{1}{4}, \; [N]_{T_0} \leq CT_0 \right) \\
        &\leq \exp\left( -\frac{c}{T_0} \right),
\end{align*}
    for some constant $c>0$, where we used Lemma \ref{lem_freedman} in the last passage.
\end{proof}
Having established Lemma \ref{lem155}, we can now use it to control the smallest eigenvalue $\lambda_{\min}(A_t)$. Since the argument is nearly identical, we only provide a sketch.

\begin{proof}[Proof of Proposition \ref{lem161}.]
    This time we employ the functional
    $$g_\beta(A) = g_\beta(A_t) = -\frac{1}{\beta}\log \Tr\left[ e^{-\beta A}\right].$$
    We set
    $$\tilde{M} = \nabla g_\beta(A) = \frac{e^{-\beta A}}{\Tr \left(e^{-\beta A}\right)}$$
    which again is a symmetric, positive semi-definite matrix of trace $1$. From Lemma \ref{lem184}, we have, for any symmetric matrix $H$,
    $$\nabla^2g_\beta(A)(H,H) \geq -\beta\Tr\left[ \tilde{M} H^2\right].$$
    Formula \eqref{eq155} yields
    \begin{align}
        dg_\beta(A) &\geq \sum_{i=1}^n \Tr[\tilde{M}H_i]dB_i - \Tr[\tilde{M}A^2]\,dt -\frac{\beta}{2}\sum_{i=1}^n \Tr[\tilde{M}H_i^2]\, dt\nonumber \\
        &\geq \sum_{i=1}^n \Tr[\tilde{M}H_i]dB_i - \|A_t\|^2_{op}dt - \frac{\beta}{2} \|A_t\|_{op}^3 \kappa_n^2 dt. \label{eq370}
    \end{align}
    Note that compared with \eqref{eq302} we got an additional drift term
    $$- \Tr[ \tilde{M}A^2]\,dt \geq - \|A_t\|^2_{op} dt $$
    coming from the fact that $A_t$ is a supermartingale. Recall that
    $$\tau = \inf\{t>0 \, ; \, \ \|A_t\|_{op}\geq 2\},$$
    and denote by $\tilde{N}_t$ the martingale
    $$\tilde{N}_t = \int_0^{t\wedge\tau} \sum_{i=1}^n \Tr[\tilde{M}H_i]dB_i.$$
    Set $$T = \frac{c_1}{\kappa_n^2\log n} \leq T_0$$
    where $c_1\leq c_0$ is to be chosen. We select $\beta = 8\log n$ and note that $g_{\beta}(A_0) = 1 - \beta^{-1} \log n$. For any $t\leq T$, integrating \eqref{eq370} yields
    \begin{align*}
        g_\beta(A_{t\wedge\tau}) &\geq \tilde{N}_t + 1 - \frac{1}{8} - 4T - 4\beta\kappa_n^2T \\
        &\geq \frac{3}{4} + \tilde{N}_t,
    \end{align*}
    where we chose the universal constant $c_1$ small enough. Thus, for any $t \leq T$,
    $$\lambda_{\min}\left(A_{t\wedge\tau}\right) \geq g_\beta(A_{t\wedge\tau}) \geq \frac{3}{4} + \tilde{N}_t.$$
    As before, we bound the quadratic variation of $\tilde{N}_t$ as
    $$[\tilde{N}]_t \leq 9\int_{0}^{t\wedge\tau} \|A_{s}\|_{op}^3 ds \leq C T.$$
    In particular, we obtain
\begin{align*}
    \PP\left( \exists\, t \in [0,T] \, ; \,
        \lambda_{\min}\left(A_{t\wedge\tau}\right) \leq \tfrac{1}{2} \right)
    &\leq \PP\left( \exists\, t \in [0,T] \, ; \, \tilde{N}_{t} \leq -\tfrac{1}{4} \right) \\
    &= \PP\left( \exists\, t \in [0,T] \, ; \,
        -\tilde{N}_{t} \geq \tfrac{1}{4},\; [\tilde{N}]_{T} \leq C T \right) \\
    &\leq \exp\left(-\frac{c}{T}\right),
\end{align*}
for some constant $c>0$. Finally, since $T \leq T_0$,
\begin{align*}
    \PP\left( \exists\, t \in [0,T] \, ; \,
        \lambda_{\min}(A_{t}) \leq \tfrac{1}{2} \right)
    &\leq \PP\left( \exists\, t \in [0,T] \, ; \,
        \lambda_{\min}(A_{t\wedge\tau}) \leq \tfrac{1}{2} \right)
        + \PP(\tau \leq T) \\
    &\leq 2 \exp\left(-\frac{c_1}{T}\right),
\end{align*}
for some constant $c_1>0$, where we used Lemma \ref{lem155} in the last passage.
This completes the proof.
\end{proof}

\noindent Department of Mathematics, Weizmann Institute of Science, Rehovot 7610001, Israel. \\
\noindent Email: \href{mailto:pierre.bizeul@weizmann.ac.il}{pierre.bizeul@weizmann.ac.il}

\medskip \noindent
School of Mathematical Sciences, Tel Aviv University, Tel Aviv 6997801, Israel; and \\ Department of Mathematics, Weizmann Institute of Science, Rehovot 7610001, Israel. \\
Email: \href{mailto:klartagb@tau.ac.il}{klartagb@tau.ac.il}

\end{document}